\documentclass[11pt]{article}

\usepackage{geometry}                % See geometry.pdf to learn the layout options. There are lots.
\usepackage[parfill]{parskip}    % Activate to begin paragraphs with an empty line rather than an indent

\usepackage[utf8]{inputenc}
\usepackage[T1]{fontenc}
\usepackage{pdfpages}
\usepackage[english]{babel}
\usepackage{tikz}
\usepackage{multirow}
\usepackage{todo}
\usepackage{amsmath,amssymb,amsfonts,amsthm}
\usepackage{xfrac} 
\usepackage{cite}
\usepackage{fullpage}

\usepackage{hyperref}
\usepackage{subfigure} 
\usepackage{booktabs}
\usepackage{bbm}
\usepackage{bm}
\usepackage[ruled,vlined]{algorithm2e}

\newcommand{\refcite}{\cite}

\newtheorem{remark}{Remark}
\newtheorem{definition}{Definition}
\newtheorem{example}{Example}
\newtheorem{proposition}{Proposition}
\newtheorem{lemma}{Lemma}
\newtheorem{theorem}{Theorem}

\usepackage{ulem}

\begin{document}

\title{A mathematical framework for dynamical social interactions with dissimulation%\\
}

\author{
Y. Saporito\textsuperscript{1,3},
M. O. Souza\textsuperscript{2,4,*}
 and 
Y. Thamsten\textsuperscript{2,5}
\\ \\
\textsuperscript{1}Escola de Matem\'atica Aplicada, Funda\c{c}\~ao Getulio Vargas, \\
Rio de Janeiro, RJ, 22250-900, Brasil
\\
\textsuperscript{2}Instituto de Matem\'atica e Estat\'istica, Universidade Federal Fluminense,\\
Niter\'{o}i, RJ, 24210-200, Brasil
\\ \\ 
Emails:\textsuperscript{3}yuri.saporito@fgv.br, \textsuperscript{4}maxsouza@id.uff.br and \textsuperscript{5}ythamsten@id.uff.br.
\\ \\
\textsuperscript{*}
}

\maketitle

\begin{abstract}
Modeling social interactions is a challenging task that requires flexible frameworks. For instance, dissimulation and externalities are relevant features influencing such systems --- elements that are often neglected in popular models. This paper is devoted to investigating general mathematical frameworks for understanding social situations where agents dissimulate, and may be sensitive to exogenous objective information. Our model comprises a population where the participants can be honest, persuasive, or conforming. Firstly, we consider a non-cooperative setting, where we establish existence, uniqueness and some properties of the Nash equilibria of the game. Secondly, we analyze a cooperative setting, identifying optimal strategies within the Pareto front. In both cases, we develop numerical algorithms allowing us to computationally assess the behavior of our models under various settings.

\paragraph*{Keywords:} Opinion dynamics; Dissimulation; Exogenous influence; Non-cooperative games; Cooperative games

\paragraph*{AMS Subject Classification:} 91D15, 49N70, 34H05
\end{abstract}

\section{Introduction}

In this paper, we investigate the following problem: what are the   dynamics that a  social system can attain as a result of interactions among the agents comprising it? Here, the subjects of our investigations are judgments --- opinions, suspicions, and doubts --- on various matters, upon which we wish to predicate dynamical and equilibria considerations. The main distinctive aspects of this work are the presence of dissimulation, and the influence of exogenous objective information on the behavior of the system's agents. We will work under the hypothesis of rationality of the agents, and the material causes for their reasoning are: (i) apprehension from their social interactions; (ii) cognitive pressures; (iii) effects  individuals observe as consequence of their (aggregate) actions on the environment they are in.

Our modeling viewpoint is similar to that of Sakoda (see   \refcite{hegselmann2017thomas}, p.p. 13-15)\footnote{More precisely, the quotation we refer to is:
\begin{quote}
    ``The checkerboard model in its present form is more of a basic conceptual framework than a model of any given social situation. It has potentiality for further elaboration to fit particular situations. As it now stands, it can be used as a visual representation of the social interaction process, relating attitudes, social interaction and social structure. It should be particularly useful in introductory courses, not only illustrating the relationship among these concepts, but also in discussing the function of models. A model is not necessarily used to predict behavior in a situation. Model building is useful in clarifying the definition of concepts and the relationship among them. Left in verbal form, concepts can be elusive in meaning, whereas computerization require precision in definition of terms. Models can be used to gain insight into basic principles of behavior rather than in finding precise predictions of results for a given social situation, and it is this function which the checkerboard model in its present form provides (\ldots). The checkerboard model provides students of social structure with a possible explanation of its dynamics.''
\end{quote}} relative to his checkerboard model. We will investigate general settings --- competitive and cooperative --- envisaging to shed some light on the problem of dynamical social interactions under dissimulation, as well as the influence of exogenous objective information upon such a system. Therefore, for concreteness, we fix a specific manner in which these aspects are incorporated. However, we advocate that our approach is flexible, allowing for modifications to capture idiosyncrasies of particular systems. Our results do not  intend themselves to be predictive; rather, we  expect they can be useful as possible starting points for explanations of some real-world social situations.

Furthermore, when postulating the specific way we would expect people to dissimulate, we take into account that they can have distinct tempers. Our classification is that individuals are either persuasive, truthful or conforming within the population. Truthful individuals always aim to express judgments which are closer to what they truly think on the matter in question. However, there are many reasons that imply dissimulate behavior. For instance, it is doubtless to say that influencing on others' judgments is a problem of major interest. Political decisions are fundamentally dependent on solving it, and the goal of any company's marketing sector is to convince people that their product is worth buying. We will refer to the type of agents trying to influence others as the persuasive ones.

There are some efforts focused on persuasive behavior in the literature. In the work   \refcite{caillaud2007consensus}, there is a study of optimal strategies for a sponsor that must convince a qualified majority to have her proposal accept. See   \refcite{che2009opinions} for an investigation of a similar problem, now involving an advisor and a decision maker. In the paper   \refcite{rusinowska2019opinion}, authors consider a framework with three persuaders. Two of them are in opposite extremes, the third of which is in the center. Each persuader targets some agent to try to exert his influence upon him, thus influencing the whole network towards his personal judgment.

Alternatively, as a result of social pressures, or due to being more passive or indecisive when making decisions, some individuals express judgments that are distinct of their actual ones simply to conform with their group. There are plenty of empirical evidence that, in many circumstances, people do behave in this way. In effect, in the popular experiment carried out in   \refcite{asch1955opinions}, people misjudged the length of vertical lines supposedly pressured by collaborators figuring as other participants. When asked, some of them confessed that their mistake was due to the discomfort of not conforming. 

We can understand truth as the adequation of the things and the mind.\footnote{According to Western metaphysical tradition, ``\textit{Veritas est adaequatio rei et intellectus}'', see St. Thomas Aquinas' \textit{De Veritate}, Q.1, A.1-4.} From this viewpoint, although social interactions effectively cause individuals to change their judgments and choices, it is relevant to assume in some way that the members of the population consider exogenous objective information, acquired via their interactions with perceived reality. In effect, facing proper indications, even an truthful person can express judgments that differ from their real ones, say for prudence. As people exchange their views, they will act upon their environment, changing it --- we propose here to assess how this modification feeds back into individual actions. In this direction, pertinent questions arise --- possibly of particular relevance nowadays --- such as whether rational social interactions can lead to consensus that we would regard as incorrect from an objective viewpoint.\footnote{E.g., when a vaccine for a given disease is proven effective, can we observe an anti-vaccine consensus?} 

There are some psychological reasons that can affect the reaction of a person to objective information. For instance, there is \textit{confirmation bias}, which refers to a tendency to favor (respectively, avoid) information that somehow agrees (respectively, go against) prior beliefs, values etc., see   \refcite{nickerson1998confirmation}. In this context, the discredit of the source of information is a related issue. Also, when expressing a judgment which deviates from her true one, the agent incurs in a psychological stress, akin to the process of \textit{cognitive dissonance}, see   \refcite{festinger1957theory}. Some that replied wrongly to the experiment of   \refcite{asch1955opinions} said that they were genuinely convinced of their wrong answer. This is a common effect of cognitive dissonance, as individuals strive for consistency.

With sociological roots in   \refcite{french1956formal,harary1959status}, DeGroot pioneered naïve learning models of opinion formation in the seminal work   \refcite{degroot1974reaching}. Taking place in a discrete-time setting, it consists of stipulating that a given agent's judgment at a period is updated by a weighted average of the ones of the previous period. In the economics literature, the authors of   \refcite{demarzo2003persuasion} employed a naïve learning model to investigate the effect of the failure of agents to account for repetitions, what they called persuasion bias. In   \refcite{golub2010naive}, they study the phenomenon of wisdom of the crowds in this model. Among further efforts on naïve learning, we mention the investigation of its relations with cooperation, see   \refcite{kirchkamp2007naive}, the analysis of the effect of Bayesian agents amidst a population of bounded rational individuals, see   \refcite{mueller2014does}, and the question of manipulation, see   \refcite{banerjee2019naive}.

A key development of the DeGroot model is the celebrated Bounded Confidence (BC) model of Hegselmann and Krause, see   \refcite{hegselmann2002opinion}, and also   \refcite{jabin2014clustering} for noteworthy mathematical advancements in this setting. The continuous-time model of naïve learning comprises a straightforward extension, see   \refcite{canuto2008eulerian,blondel2010continuous}. There are many advances based on the BC modeling setups. In   \refcite{hegselmann2006truth,hegselmann2009deliberative}, they regard individuals to be sensible to external information. When considering the action of a leader upon the population, some works taking following a control-theoretical perspective are   \refcite{borzi2015modeling,wongkaew2015control,dietrich2017control}. The recent work   \refcite{han2019opinion} presents results in a modified BC model with stubbornness as a type of persistence. The paper   \refcite{bauso2016opinion} regards a Mean-Field Game (MFG) model account for external disturbances and random noise, in such a way that, in a certain sense, the resulting strategies are robust with respect to uncertainty. We also refer to the paper   \refcite{degond2017continuum} for an MFG model studying long-time dynamics of an opinion formation framework. In these references, authors assume that the expressed judgments coincide with the real ones. A recent advance, in this context, concerns the BC model, namely, the effect of mis- and disinformation on it, see   \refcite{douven2020mis}.

The closest works to the present one in the literature are \refcite{buechel2015opinion} and \refcite{etesami2018influence}. In the former, agents can be conforming, counter-conforming or truthful; in this connection, see also \refcite{bala2001conformism} for an alternative approach to conformity. They build upon the DeGroot model, whence it is a discrete-time framework. Moreover, judgments are one-dimensional, the optimization determining the expressed judgment of an agent is static, and their model does not include the effect of objective information in the dynamics of the population. In the latter, they propose a number of discrete-time game dynamics where the expressed judgment of each agents can be either binary or come from a continuum. 
In this dynamics, the behavior of agents range from manipulative to conformist.

Here, we consider a continuous-time model akin to the BC one. We work under the framework of control theory, stipulating performance criteria for each of the individuals determining their behavior as a result of some notion of equilibrium for the corresponding game. In this framework, we can allow for players to react to external signals. 

We now mention a few works treating problems that are related to ours from a distinct modeling viewpoint. In   \refcite{ellison1993rules,ellison1995word}, we find alternative approaches to social learning. The work   \refcite{blume2015linear} comprises a general study of static linear models. The paper   \refcite{calvo2004effects} develops an analysis of unemployment via a mechanism of information exchange within a network. Two recent approaches to social learning are   \refcite{arieli2019multidimensional}, with random networks, and   \refcite{mossel2020social}, which accounts for the reaction to learning from private signals with a focus on static equilibria. Under Bayesian learning and influence of external information, in the work   \refcite{rosenberg2009informational} there is an analysis of emergence of consensus.

Our model comprises a finite population of strategically interacting individuals. Each player has a multidimensional true judgment on a variety of matters, and chooses to express another one that can possibly deviate from it. The expressed judgment is determined by each player according to her objective criteria; see   \refcite{bailo2018pedestrian} for the application of a related idea in pedestrian dynamics. In this context, the agents assess their performance via functionals that are constituted of two parts. One of them regards differences between the expressed judgment and two quantities: the true judgment, akin to a cognitive dissonance stress; the average population judgment, which models the behavior (either persuasive or conforming) of the corresponding agent. The other piece forming the functional is through where we introduce the effect of objective information in the game. The state variables evolve in time as a result of the interaction of each player with the expressed judgments profile of the population.

We consider two distinct settings. The first one is a competitive game. We prove the existence of Nash equilibrium, and that it is in fact unique under suitable assumptions. This is a natural notion of equilibrium, e.g., if we think agents are continuously debating and trying to convince one another, in a accordance to what suits their nature. We provide some numerical illustrations to showcase the rather rich dynamics we obtain resulting from the various possible configurations, departing from the same initial judgments profile. The second framework we investigate is the one in which players cooperate. We are able to characterize the Pareto front, and also numerically illustrate the resulting strategies that are optimal in this sense. Understanding cooperative formation can shed light, e.g., in the study of legislative bargaining, see   \refcite{gomes2005dynamic}. The latter setting seems to be reasonable for making conceptual considerations on this problem.

We organize the remainder of this paper as follows. In Section \ref{sec:model}, we present the technical aspects of the model, such as the evolution of the state variables, and the performance criteria of the individuals in the population. We also provide some well-posedness results that will be of major importance in the work. Then, we consider the competitive setting in Section \ref{sec:Competitive}, characterizing the appropriate equilibria, and discussing the asymptotic behavior of them. Then, in Section \ref{sec:Numerics} we provide a numerical algorithm of the equilibria we previously found, and also present many experiments of possible configurations that we can attain. In Section \ref{sec:Coalitional} we proceed in a similar manner, but supposing that agents among the population cooperate. Lastly, we present  our concluding remarks in Section \ref{sec:Conclusions}.

\section{The model} \label{sec:model}

\subsection{Presentation of the model}

Let us consider a population of $N > 1$ agents labeled by $i \in \mathcal{N} := \left\{1,...,N\right\}.$ Interactions among players will occur throughout a time horizon $\left[0,T\right],$ for a fixed $T>0.$ For $i \in \mathcal{N},$ we represent the actual judgments of player $i$ at time $t \in \left[0,T\right]$ by a multi-dimensional vector $x_i(t) \in \mathbb{R}^d,$ $d\geqslant 1,$ whereas we denote the judgment this agent chooses to express by $\omega_i.$ For instance, we can regard $x_i$ and $\omega_i$ as one-dimensional, thus denoting the real and expressed judgments, respectively, of agent $i$ about a situation containing two opposing extremes. Denoting the radical positions by $X$ and $Y,$ upon proper scaling, we can consider that person $i$ holding position $X$ (resp., $Y$) is such that $x_i = 0$ (resp., $x_i=1$), whereas we would represent an extremist advocate of position $X$ by $\omega_i = 0$ (resp., $\omega_i = 1$). In this setting, we can interpret people located in positions in between $0$ and $1$ in the obvious way.

In general, we assume the following dynamics for the system: 
\begin{equation} \label{eq:BasicModel}
    \begin{cases}
    \dot{x}_i(t) = \sum_{j=1}^N K_{ij}\left( \omega_j(t) - x_i(t)\right),\, 0 < t < T,\\
    x_i(0) = x_{i0}.
    \end{cases}
\end{equation}
Above, the functions $K_{ij}$ are the interaction kernels. Henceforth, we make the subsequent assumptions on them.
\begin{itemize}
    \item[\textbf{(A)}] For each $(i,j) \in \mathcal{N}^2,$ we have $K_{ij}(z)= a_{ij}(z)z,$ for a $C^2_b$ non-negative function $a_{ij} : \mathbb{R}^d \rightarrow \mathbb{R} .$
\end{itemize}

Regarding the expressed judgments, we assume $\omega_i \in \mathcal{A}_i,$ for an admissible control set of the form
$$
\mathcal{A}_i := \left\{ \omega_i \in L^2(0,T): \omega_i(t) \in A_i \text{ for almost every } t\in \left[0,T\right]\right\},
$$
where $A_i \subseteq \mathbb{R}^d.$ We write $\mathcal{A} := \Pi_{i=1}^N \mathcal{A}_i.$ Whenever we want to emphasize $T$ in the definition of $\mathcal{A},$ we will write $\mathcal{A}_T \equiv \mathcal{A}.$ Moreover, we fix the subsequent assumption on the action spaces: 
\begin{itemize}
    \item[\textbf{(B)}] The set $A_i$ is a closed and convex subset of $\mathbb{R}^d,$ and there exists $R>0$ such that, for every $i \in \mathcal{N},$ we have $A_i \subseteq \left[-R,R\right]^d.$  
\end{itemize}
\begin{remark} \label{rem:projection}
Hereafter, we will denote the $L^2(0,T)^N-$projection over $\mathcal{A}_i$ by $P_{\mathcal{A}_i}.$
\end{remark}

Thus, instead of reacting to the real judgments of other players, we consider that agent $i$ interacts with the profile of expressed judgments. We advocate that this assumption is more realistic, for we do not expect that player $i$ would be able to identify the true opinions of the others, unless they deliberately choose to express them, and are capable of doing so effectively. This does not mean that player $i$ does not acknowledge at all the real judgments of the other players, as these are taken into account in the formation of $\omega_j,$ for each $j \in \mathcal{N},$ as we will later see in our main results. Thus, the judgments of player $i$ will have an evolution indirectly impacted by $x_j,$ for $j\neq i,$ viz., through the choice that player $j$ makes for $\omega_j.$ 

We also point out our inclusion of the term $K_{ii}\left( \omega_i - x_i \right)$ in the dynamics \eqref{eq:BasicModel}. This represents the effect that, when emitting a judgment that is not the true one of the agent, a tension is created. Consequently, through this term, the actual judgment of this agent ought to be pushed towards the dissimulated one. This is an instance in which we introduce an effect akin to cognitive dissonance in our framework.

We assume that all persons within the population are rational. The way that they will select their expressed judgments is founded on objective criteria. More precisely, for $i \in \mathcal{N},$ the agent $i$ assigns a functional $J_i : \mathcal{A} \rightarrow \mathbb{R}$ as follows
\small
\begin{equation} \label{eq:ObjCriteria}
    J_i(\omega_i;\boldsymbol{\omega}_{-i}) := \int_0^T \left[ \frac{\left(1-\delta_i\right)}{2}|\omega_i(t) - x_i(t)|^2 + \frac{\delta_i}{2}\left|\overline{\omega}(t) - x_{i}(t) \right|^2 + \zeta_i \lambda\left( \overline{x}(t)\right) \right]\,dt,
\end{equation}
\normalsize
where we employed the notations 
$$
\boldsymbol{\omega}_{-i} := \left( \omega_1,\ldots,\omega_{i-1},\omega_{i+1},\ldots,\omega_N \right)^\intercal,
$$
$$
\left(\omega_i;\boldsymbol{\omega}_{-i}\right) := \boldsymbol{\omega},
$$
\begin{equation} \label{eq:AvgOmegaDefn}
    \overline{\omega}(t) := \frac{1}{N} \sum_{j=1}^N \omega_j(t),
\end{equation}
i.e., the quantity $\overline{\omega}$ figuring in \eqref{eq:ObjCriteria} denotes the average expressed judgment of the population, whereas $\overline{x}$ is their true counterpart 
\begin{equation} \label{eq:AvgXDefn}
    \overline{x}(t) := \frac{1}{N} \sum_{j=1}^N x_j(t).
\end{equation} 

Let us now discuss how we structured the functional \eqref{eq:ObjCriteria}. It is of the form
\begin{equation} \label{eq:ObjCriteria_rewritten}
    J_i\left( \boldsymbol{\omega} \right) = \frac{1}{2}\widetilde{J}_i\left( \boldsymbol{\omega} \right) + \zeta_i I\left( \boldsymbol{\omega} \right),
\end{equation}
with
$$
\widetilde{J}_i\left( \boldsymbol{\omega} \right) := \left(1-\delta_i\right) \int_0^T |\omega_i(t) - x_i(t)|^2\,dt + \delta_i \int_0^T \left|\overline{\omega}(t) - x_{i}(t) \right|^2\,dt, 
$$
and
$$
I\left( \boldsymbol{\omega} \right) := \int_0^T \lambda\left( \overline{x}(t)\right) \,dt.
$$
We begin by making some considerations on $\widetilde{J}_i.$

The part of $\widetilde{J}_i$ comprising
\begin{equation} \label{eq:Part1Obj}
    \int_0^T \left| x_i(t) - \omega_i(t) \right|^2\,dt
\end{equation}
is another instance in which we model cognitive dissonance. The piece 
\begin{equation} \label{eq:Part2Obj}
    \int_0^T \left| x_i(t) - \overline{\omega}(t) \right|^2\,dt
\end{equation}
brings into \eqref{eq:ObjCriteria} the deviation between the actual judgment of the agent and the global average judgment. For $0 < \delta_i <1$ and for each $\boldsymbol{\omega}_{-i} \in \mathcal{A}_{-i} := \Pi_{j\neq i}\mathcal{A}_j,$ we can see $\omega_i \in \mathcal{A}_i \mapsto \widetilde{J}_i\left( \omega_i; \boldsymbol{\omega}_{-i} \right) \in \mathbb{R}$ as the functional whose minimum is a Pareto optimal strategy for the bi-objective problem \eqref{eq:Part1Obj}-\eqref{eq:Part2Obj} (in the $i-$th direction) --- we will take back to this discussion in Section \ref{sec:Coalitional}. Regarding the parameter $\delta_i \in \left]-\infty,1\right[,$\footnote{This asymmetric interval for the parameters $\delta_i$ results from our particular parameterization of the model.} we observe that it represents the persuasiveness/conformity level of the agent. Thus, an agent who seeks to convince others of having the same judgment as her true one, has $\delta_i > 0,$ as $\delta_i$ enters in \eqref{eq:ObjCriteria} directly proportionally to \eqref{eq:Part1Obj}. Similarly, people who tend to conform to the average populations' judgment have $\delta_i < 0,$ this effect being more intense the larger $|\delta_i|$ is. Finally, having $\delta_i = 0$ is proper of a truthful person, as such an agent only values expressing a judgment close to her real one, being to an extent indifferent to the average population expressed judgment.

Here, we stipulate that people envisage to interact with the average manifested judgment of the whole population. This is in distinction to other works in the literature, such as   \refcite{buechel2015opinion}, in which agents only regard a local average (in the bounded confidence sense). Employing similar techniques as the ones we will present here, we could consider alternatives, such as replacing $\overline{\omega}$ in \eqref{eq:ObjCriteria} by
$$
\overline{\omega}_i(t) := \sum_{j=1}^N a_{ij}\left( \omega_j(t) - x_i(t) \right)\omega_j(t).
$$
For simplicity, from now on, we stick to \eqref{eq:AvgOmegaDefn} and \eqref{eq:AvgXDefn}. 

Let us now consider the $I$ component in \eqref{eq:ObjCriteria_rewritten}. The function $\lambda : \mathbb{R}^{d} \rightarrow \mathbb{R}$ is supposed to encode exogenous objective information into the individual criterion. Here, the way we choose to model this is to assume that, as players' actions impact reality, there will be a feedback effect perceived by them through the quantity $\lambda.$ However, agents do not necessarily have the same sensitivity to the same information. This is the reason why we introduce the parameter $\zeta_i.$ In this manner, there is a balance, through the latter constant, between the willingness of player $i$ to persuade/conform, and their reaction to real evidence that is faced as consequence of the aggregate interaction between the population and the environment where they are situated in. Regarding $\lambda,$ we suppose:

\begin{itemize}
    \item[\textbf{(C)}] The function $\lambda$ is of class $C^2_b.$ 
\end{itemize}

We proceed to give an example consisting of the main motivation for taking $I$ in the form we exposed in \eqref{eq:ObjCriteria_rewritten}.

\begin{example} \label{ex:Example}
Let us consider the one-dimensional setting, say with judgments varying over the action space $\left[0,\,1\right].$ We consider that objective information corroborates the choice of position $1,$ in such a way that the adoption of position $0$ by the population leads to a worst outcome in terms of the third summand within the integral figuring in \eqref{eq:ObjCriteria}. For concreteness, we propose here
$$
\lambda(\overline{x}) := \lambda_{0} + \lambda_{1}\left( 1 - \overline{x} \right) ,
$$
with $\lambda_0,\lambda_1 > 0.$ Let us designate the number of occurrences of undesirable events that would be mitigated if people were to adopt position $1$ by $N^{\boldsymbol{\omega}}.$ We assume that $N^{\boldsymbol{\omega}}$ is an nonhomogeneous Poisson point process with intensity $\lambda\left(\overline{x}\right),$ for a given profile $\boldsymbol{\omega} \in \mathcal{A}$ (where  $\overline{x}$ results from \eqref{eq:AvgXDefn}, for true judgments $\boldsymbol{x} = \left(x_1,...,x_N\right)^\intercal$ given by \eqref{eq:BasicModel}). The intensity of $N^{\boldsymbol{\omega}}$ would be minimal (equal to $\lambda_0$) if $x_i \equiv 1,$ for all $i \in \mathcal{N}.$ In general, we observe that minimizing the expected value of $N^{\boldsymbol{\omega}}_T$ amounts to minimizing
$$
\mathbb{E}\left[ N^{\boldsymbol{\omega}}_T \right] = \int_0^T \lambda\left(\overline{x}(t)\right)\,dt = I\left( \boldsymbol{\omega} \right).
$$
\end{example}

\subsection{On the well-posedness of the model}

The subsequent results are devoted to establishing basic properties of the model \eqref{eq:BasicModel}. We will first prove existence and uniqueness of a solution $\boldsymbol{x},$ for each initial datum $\boldsymbol{x}_0 = \left(x_{01},\ldots, x_{0N}\right)^\intercal$ and each given profile of strategies $\boldsymbol{\omega} \in \mathcal{A}.$ Then, we will prove a continuity property of the true judgments in terms of the expressed judgments. We emphasize that, although we consider at first the behavior $\boldsymbol{x}$ in terms of $\boldsymbol{\omega},$ the equilibria we will investigate will actually involve a fixed point relation connecting these two quantities. We address these questions in Sections \ref{sec:Competitive} and \ref{sec:Coalitional}. We proceed to provide the definition of the solution concept we consider for the ODEs of Eq. \ref{eq:BasicModel}.

\begin{definition}
Given $\boldsymbol{\omega} \in \mathcal{A},$ we say that a continuous function $\boldsymbol{x} : \left[0,T\right] \rightarrow \mathbb{R}^{dN}$ is a solution of \eqref{eq:BasicModel} if, for each $i\in \mathcal{N}$ and $t \in \left[0,T\right],$ we have 
$$
x_i(t) = x_{0i} + \int_0^t \sum_{j=1}^N K_{ij}\left(\omega_j(u) - x_i(u)\right) \,du.
$$
\end{definition}
Thus, we resort to the concept of solutions in the sense of Caratheodory. We have the following result on existence and uniqueness, which we can prove, under assumption $(\textbf{A}),$ using the same methodology as in Chapter $2$ of   \refcite{teschl2012ordinary} --- see Theorems 2.5 and 2.17 therein.

\begin{proposition} \label{prop:ExistenceAndUniqueness}
For each $\boldsymbol{\omega} \in \mathcal{A},$ the model \eqref{eq:BasicModel} admits a unique (globally defined) solution.
\end{proposition}

In a similar fashion, we can use the same techniques allowing us to prove Proposition \ref{prop:ExistenceAndUniqueness} to obtain the subsequent result.

\begin{proposition} \label{prop:ContinuityOnOmega}
Let us assume that $\boldsymbol{\omega},\widetilde{\boldsymbol{\omega}} \in \mathcal{A}$ are two admissible expressed judgments, as well as $\boldsymbol{x}_0,\widetilde{\boldsymbol{x}}_0 \in \mathbb{R}^d$ are two initial configurations of true judgments. Let us denote by $\boldsymbol{x},\widetilde{\boldsymbol{x}}$ the solutions of \eqref{eq:BasicModel} corresponding to the expressed judgments-initial datum couples $\left( \boldsymbol{\omega},\boldsymbol{x}_0 \right)$ and $\left( \widetilde{\boldsymbol{\omega}},\widetilde{\boldsymbol{x}}_0\right)$, respectively. Then, for each $i \in \mathcal{N},$
$$
\sup_{0\leqslant t \leqslant T}\max_j|x_j(t)-\widetilde{x}_j(t)| \leqslant Ce^{C T}\left( \left|x_{0i} - \widetilde{x}_{0i}\right| + \sum_{j=1}^N \|\omega_j - \widetilde{\omega}_j\|_{L^1(0,T)} \right),
$$
where $C$ is independent of $T.$
\end{proposition}

In the one-dimensional case, it is straightforward to derive from the component-wise uniqueness of \eqref{eq:BasicModel} we proved in Proposition \ref{prop:ExistenceAndUniqueness} the following monotonicity property of the individual trajectories. 

\begin{lemma} \label{lem:Monotonicity}
Let us assume $d=1,$ and that $K_{ij}$ is independent of $(i,j),$ i.e., $K_{ij} \equiv K_{i^\prime j^\prime},$ for every $i,\,i^\prime,\, j,\,j^\prime \in \mathcal{N}.$ Then, for any given profile of expressed judgments $\boldsymbol{\omega},$ the relation $x_{i,0} \leqslant x_{j,0}$ implies $x_i(t)\leqslant x_j(t),$ for every $t \in \left[0,T\right].$
\end{lemma}

Throughout this whole paper, the next property will be key for the study of the concepts of equilibrium that we will consider. It asserts that solutions will be confined to a fixed hyper-cube, as long as the initial condition originates within it.

\begin{lemma}
Let us suppose that the initial set of judgments satisfy $\left|x_{i,0}\right|\leqslant R$ (cf. assumption $\textbf{(B)}$). Then, for every $t \in \left[0,T\right]$, we have $\left|x_i(t)\right|\leqslant R.$
\end{lemma}
\begin{proof}
In view of our assumption $\textbf{(B)},$ we have $\left|\omega_i\right|\leqslant R.$ If $\left|x_i(t)\right| = R,$ then at this time $t$ we have
\begin{align*}
    \frac{d}{du}\left|x_i(u)\right| \Big|_{u=t} &=  \frac{1}{2|x_i(t)|} x_i(t) \cdot \sum_{j=1}^N K_{ij}\left(\omega_j(t) - x_i(t)\right) \\
    &= \frac{1}{2}\sum_{j=1}^N a_{ij}\left(\omega_j(t) - x_i(t)\right) \left( \frac{x_i(t)\cdot \omega_j(t)}{R}  -  R \right) \\
    &\leqslant \frac{1}{2}\sum_{j=1}^N a_{ij}\left(\omega_j(t) - x_i(t)\right)\left( \left|\omega_j(t)\right| - R \right) \leqslant 0.
\end{align*}
Therefore, we indeed have $\left|x_i(t)\right|\leqslant R,$ for every $t \in \left[0,T\right].$
\end{proof}

\section{The non-cooperative game} \label{sec:Competitive}

\subsection{Nash equilibria}

In this section, we investigate equilibria resulting from the assumption that individuals are rational, and do not seek cooperation. Thus, we are concerned with strategies constituting a Nash equilibrium, which we define as follows.
\begin{definition}
A strategy profile $\boldsymbol{\omega}^* = (\omega_1,...,\omega_N)^\intercal \in \mathcal{A}$ is an open-loop Nash equilibrium if, and only if, for each $i \in \mathcal{N}$ and each $\omega_i \in \mathcal{A}_i,$ the relation
$$
J_i(\omega_i^*;\boldsymbol{\omega}_{-i}^*) \leqslant J_i(\omega_i;\boldsymbol{\omega}_{-i}^*)
$$
holds.
\end{definition}

\subsection{Variational approach}

Our next step in the analysis of Nash equilibria is to derive necessary conditions that such a strategy has to satisfy. We employ techniques of variational analysis, characterizing the optimal strategies as solutions to a coupled ODE system. Moreover, this system is augmented with suitable adjoint parameters. We precisely state and prove this result in the sequel.

Theorem \ref{thm:NecessaryCondnNE} below appears in many disguises in the literature, see \refcite{basar2018handbook} and references therein. However, since it will be convenient to refer to steps of the proof later on, we include a complete proof.

\begin{theorem} \label{thm:NecessaryCondnNE}
A Nash equilibrium $\boldsymbol{\omega}^*$ must solve the fixed point equation
\begin{equation} \label{eq:FixedPoint}
    \boldsymbol{\omega}^*=\Phi\left[\boldsymbol{\omega}^*\right],
\end{equation}
where the mapping $\Phi : L^2(0,T)^N \rightarrow \mathcal{A}$ is defined as\footnote{For the definition of $P_{\mathcal{A}_i},$ see Remark \ref{rem:projection}.}
\begin{align} \label{eq:DefnOfPhi}
    \begin{split}
                \Phi[\boldsymbol{\omega}]_i = P_{\mathcal{A}_i} &\left( x_i- \frac{\delta_i}{( 1-\delta_i ) N}\left( \overline{\omega} - x_{i} \right) \right.\\
                &\left. \hspace{1.0cm}- \frac{1}{(1-\delta_i)}\sum_{j=1}^N K_{ji}^\prime\left( \omega_i - x_j \right)^\intercal \varphi_{ji} \right),
    \end{split}
\end{align}
with $\varphi_{ji}$ being the solutions of 
\small
\begin{equation} \label{eq:AdjointParameters}
    \begin{cases}
        -\dot{\varphi}_{ji}(t) = - \sum_{l=1}^N K_{jl}^\prime\left(  \omega_l(t) - x_j(t)\right)^\intercal\varphi_{ji}(t) \\
        \hspace{1.5cm}+ \boldsymbol{1}_{i=j}\left\{ x_i(t) -\left[ \delta_i \overline{\omega}(t) + \left(1-\delta_i\right) \omega_i(t)\right] \right\} + \frac{\zeta_i}{N} \nabla\lambda \left(\overline{x}(t)\right) , \, 0\leqslant t \leqslant T,\\
        \varphi_{ji}(T) = 0.
    \end{cases}
\end{equation}
\normalsize
\end{theorem}
\begin{proof} 
We compute the G\^ateaux derivative 
\begin{align*}
    \big\langle D_i J_i(\omega_i;\boldsymbol{\omega}_{-i}), v^i \big\rangle =&\, \int_0^T \left[ \delta_i(\overline{\omega}(t) - x_i(t))\cdot\left( \frac{1}{N}v^i(t)- y^i(t)\right) \right.\\
    &\left. \hspace{0.9cm}+\left(1-\delta_i\right)\left( \omega_i(t) - x_i(t)\right) \cdot \left( v^i(t) - y^i(t)\right) \right.\\
    &\left.\hspace{0.9cm} + \frac{\zeta_i}{N} \sum_{j=1}^N\nabla \lambda\left( \overline{x}(t)\right) \cdot y^{ji}(t) \right]\,dt,
\end{align*}
where
\footnotesize
\begin{equation} \label{eq:ODEforY}
    \begin{cases}   
        \dot{y}^{ji}(t) = - \sum_{l=1}^N K_{jl}^\prime\left(  \omega_l(t) - x_j(t)\right) y^{ji}(t) + K_{ji}^\prime\left( \omega_i(t) - x_j(t) \right)v^i(t),\, 0\leqslant t \leqslant T,\\
        y^{ji}(0) = 0,
    \end{cases}
\end{equation}
\normalsize
and $y^i:=y^{ii}.$ Let us introduce the adjoint parameters $\left\{ \varphi_{ji} \right\}_{i,j}$ as the solutions of \eqref{eq:AdjointParameters}. In this way, we deduce
\footnotesize
\begin{align*}
    \big\langle D_iJ_i(\omega_i;\boldsymbol{\omega}_{-i}),&\,v^i\big\rangle \\
    &= \int_0^T \left\{ v^i(t)\left[\left(1-\delta_i\right)\left( \omega_i(t) - x_i(t) \right) + \frac{\delta_i}{N} \left( \overline{\omega}(t) - x_{i}(t) \right) \right] \right.\\
    &\hspace{1.0cm} + \left. \sum_{j=1}^N y^{ji}(t)\left\{ \boldsymbol{1}_{i=j}\left[ x_{i}(t) - \delta_i \overline{\omega}(t) + \left(1-\delta_i\right) \omega_i(t) \right] + \frac{\zeta_i}{N} \nabla \lambda \left( \overline{x}(t)\right) \right\} \right\}\,dt \\
    &= \int_0^T \left\{ v^i(t)\left[\left(1-\delta_i\right)\left( \omega_i(t) - x_i(t) \right)+ \frac{\delta_i}{N} \left( \overline{\omega}(t) - x_{i}(t) \right) \right] \right.\\
    &\left. \hspace{1.8cm} +  \sum_{j=1}^N y^{ji}(t)\left[ -\dot{\varphi}_{ji}(t) + \sum_{l=1}^N K_{jl}^\prime\left( \omega_l(t) - x_j(t) \right)^\intercal\varphi_{ji}(t) \right]\right\}\,dt \\
    &= \int_0^T \left\{ v^i(t)\left[\left(1-\delta_i\right)\left( \omega_i(t) - x_i(t) \right) + \frac{\delta_i}{N} \left( \overline{\omega}(t) - x_{i}(t) \right) \right] \right.\\
    &\left. \hspace{1.8cm} + \sum_{j=1}^N \varphi_{ji}(t)\left[ \dot{y}^{ji}(t) +  \sum_{l=1}^N K_{jl}^\prime\left( \omega_l(t) - x_j(t) \right)y^{ji}(t) \right]\right\}\,dt \\
    &= \int_0^T v^i(t) \left\{\left(1-\delta_i\right)\left( \omega_i(t) - x_i(t) \right)  + \frac{\delta_i}{N} \left( \overline{\omega}(t) - x_{i}(t) \right) \right. \\
    &\left.\hspace{1.8cm} + \sum_{j=1}^N K_{ji}^\prime\left( \omega_i(t) - x_j(t) \right)^\intercal \varphi_{ji}(t) \right\}\,dt.
\end{align*}
\normalsize
From this, the result promptly follows.
\end{proof}

Now, two remarks concerning the result we presented in Theorem \ref{thm:NecessaryCondnNE} are in order.

\begin{remark}
For a persuasive agent (that is, with a positive persuasion parameter $\delta_i$), the formulae \eqref{eq:FixedPoint}-\eqref{eq:DefnOfPhi} show that she adds a term proportional to the difference between her actual and the average stated judgment throughout the population. Thus, such persuasive agents are more likely to radicalize, envisioning to convince others. This effect becomes more intense the closer their persuasiveness parameter is to one. If, however, the agent is conforming, then she will add to her true judgment a term leading her in the direction of $\overline{\omega}.$ For instance, when $\delta_i = -N/(N+1),$ the first two terms within the projection of \eqref{eq:DefnOfPhi} add up to $\overline{\omega}.$ We can interpret this as if such an agent were always willing to reinforce the average judgment she perceives out of their peers, whatever her actual judgment is. The latter claim is, of course, disregarding the influence stemming from the adjoint parameters, constituting the third term in \eqref{eq:DefnOfPhi}. In fact, the remaining element in formulae \eqref{eq:FixedPoint}-\eqref{eq:DefnOfPhi} comprises effects captured by the adjoint parameters. 
\end{remark}

\begin{remark}
The adjoint parameters $\varphi_{ji},$ with $j\neq i,$ can only be non-vanishing if $\zeta_i \neq 0$ and if $\lambda$ is not constant. In this way, we see that these bring the exogenous objective information into each individual players' consideration. This occurs rather indirectly, as we would expect from \eqref{eq:ObjCriteria}. Regarding $\varphi_{ii},$ both $\zeta_i$ and $\lambda$ are still present, but there is also the influence of the tension between $x_i$ and the weighted average $\delta_i \overline{\omega} + (1-\delta_i)\omega_i.$ Moreover, we remark that $\varphi_{ii}$ appears multiplied by $K_{ii}^\prime,$ which we used to dynamically model cognitive dissonance.
\end{remark}

We now turn to sufficiency. The co-state Equations \eqref{eq:AdjointParameters} are indeed equivalent to the necessary conditions yielded by the maximum principle as found in \refcite{bacsar1998dynamic} and \refcite{bryson2018applied}; see also the recent review in \refcite{basar2018handbook}. However, standard results on sufficiency of the maximum principle require additional assumptions, as for instance, convexity. If our dynamics \eqref{eq:BasicModel} were linear in $\omega_j - x_i$ and the exogenous objective information function, $\lambda$, were convex, sufficiency  would follow from those standard results. However, in our setting, because of the nonlinear nature of the kernel $K_{ij}$, we need to resort to a more analytical approach. In particular, even with the linear dynamics, our approach is able to cater for a non-convex $\lambda$.

We argue that a natural space to seek these equilibria is in the space consisting of continuous functions with appropriately constrained action spaces, viz. according to assumption $\textbf{(B)}.$ We first prove that, for sufficiently small terminal time horizons, the mapping $\Phi$ admits a unique fixed point in the space we just described. Then, we provide two results for larger times frames. The first one just asserts existence of a possibly discontinuous Nash equilibrium, and the second one guarantees the existence of a continuous one. To prove the first, we develop a continuation method. The second one follows from the first by means of compactness arguments. Before entering in this two major theorems, we remark the following stability property for the adjoint system. We can prove it using the same techniques we referred to in Propositions \ref{prop:ExistenceAndUniqueness} and \ref{prop:ContinuityOnOmega}.

\begin{lemma} \label{lem:EstimatesForAdjSystem}
Let us write $\varphi^{\omega_i}_{ji}$ to denote the adjoint state $(i,j)$ corresponding to $\omega_i \in \mathcal{A}_i$ (the remaining $\boldsymbol{\omega}_{-i}$ being held fixed), and let us assume $T\leqslant 1.$ We have
$$
\sup_{0\leqslant t \leqslant T}|\varphi_{ji}^{\omega_i}(t)| \leqslant C \left[ \boldsymbol{1}_{i=j}\left( \| x_i\|_{L^1(0,T)} + \max_l \|\omega_l\|_{L^1(0,T)} \right) +  1 \right]
$$
and
\scriptsize
\begin{align*}
    \sup_{0\leqslant t \leqslant T}|\varphi_{ji}^{\omega_i}(t) - \varphi_{ji}^{\widetilde{\omega}_i}(t)| \leqslant C\left(1 + \sup_{0 \leqslant t \leqslant T} \left|\varphi_{ji}^{\omega_i}(t)\right|\right)\left( \max_{l} \left\| \omega_l - \widetilde{\omega}_l \right\|_{L^1(0,T)} + \max_{l} \left\| x_l - \widetilde{x}_l \right\|_{L^1(0,T)} \right).
\end{align*}
\end{lemma}
\normalsize

We are ready to provide our local-in-time existence result.

\begin{theorem} \label{thm:LocalResult}
If $ \frac{|\delta_i|}{|1-\delta_i|N} < 1 $ and $\max_i\left|x_{0,i}\right|\leqslant R,$ then there exists $\tau_0 = \tau_0\left( N,\, K_{ij},\,\delta_i,\, R \right)>0$ for which $T\leqslant \tau_0$ implies that the mapping $\Phi$ has a unique fixed point $\boldsymbol{\omega}^* \in C\left( \left[0,T\right] \right)^N \cap \mathcal{A}_{T}.$ Moreover, $\boldsymbol{\omega}^*$ is a Nash equilibrium.
\end{theorem}
\begin{proof}
Let us consider $\boldsymbol{\omega},\widetilde{\boldsymbol{\omega}} \in C\left( \left[0,T\right]\right).$ According to the estimates in Lemma \ref{lem:EstimatesForAdjSystem}, for $T\leqslant 1,$ we derive
$$
\|\Phi[\boldsymbol{\omega}] - \Phi[\widetilde{\boldsymbol{\omega}}] \|_\infty \leqslant \left[ CT +\frac{\delta_i}{(1-\delta_i)N}\right]\|\boldsymbol{\omega} - \widetilde{\boldsymbol{\omega}}\|_\infty,
$$
where $C$ is independent of $T.$ Thus, as long as $|\delta_i|< N|1-\delta_i| ,$ we can choose $T$ small enough so as to make $\Phi$ a contraction. Since the image of $\Phi$ is contained in $\mathcal{A}_T,$ any fixed point necessarily belongs to this space, thus being admissible.

Let us show that, possibly making $T$ smaller, this unique fixed point is in fact a Nash equilibrium. We consider second-order conditions:

\footnotesize
\begin{align} \label{eq:SecOrd1}
  \begin{split}
    \left\langle D_i^2 J_i( \omega_i; \boldsymbol{\omega}_{-i}), (v^i,\,v^i) \right\rangle =& \int_0^T v^i(t) \cdot \left\{ (1-\delta_i)(v^i(t) - y_i(t)) + \frac{\delta_i}{N}\left( \frac{v^i(t)}{N} - y_i(t) \right) \right.\\
    &\,+ \sum_{j=1}^N \left[ \left( \boldsymbol{1}_{i=j}v^i(t) - y_{ji}(t) \right)^\intercal K_{ji}^{\prime\prime}(\omega_j(t) - x_j(t))^\intercal \varphi_{ji}(t) \right.\\
    &\left.\left.\, + K_{ji}^\prime(\omega_j(t) - x_i(t))^\intercal \psi_{ji}(t) \right] \right\}\,dt ,
  \end{split}
\end{align}
\normalsize
where we described $y_{ij}$ in \eqref{eq:ODEforY}, and $\psi_{ji} := \left\langle D_i \varphi_{ji}, v^i\right\rangle$ solves
\footnotesize
\begin{equation} \label{eq:DerivAdjointParameters}
    \begin{cases}
        -\dot{\psi}_{ji}(t) = - \sum_{l=1}^N K_{jl}^\prime\left(  \omega_l(t) - x_j(t)\right)^\intercal\psi_{ji}(t) + \boldsymbol{1}_{i=j}\left\{ y_{ji}(t) -\left[ \frac{\delta_i}{N} + \left(1-\delta_i\right) \right] v_i(t) \right\}  \\
        \hspace{1.2cm} - \sum_{l=1}^N \left(v_i(t) \boldsymbol{1}_{l=i} - y_{ji}(t) \right)^\intercal K_{jl}^{\prime\prime}\left( \omega_l(t) - x_j(t) \right)^\intercal \varphi_{ji}(t) \\
        \hspace{1.2cm} + \frac{\zeta_i}{N^2} D^2\lambda \left(\overline{x}(t)\right) \cdot \sum_{l=1}^N y_{j,l}(t),\, 0 < t < T,\\
        \psi_{ji}(T) = 0.
    \end{cases}
\end{equation}
\normalsize
As in Lemma \ref{lem:EstimatesForAdjSystem}, if $T\leqslant 1,$ then we can prove that
$$
|y_{ij}(t)|\leqslant CT\|v^i\|_{L^1(0,T)} \leqslant CT\|v^i\|_{L^2(0,T)},
$$
and likewise
$$
|\psi_{ij}(t)|\leqslant CT\|v^i\|_{L^2(0,T)}.
$$
Thus, from \eqref{eq:SecOrd1}, we derive
\small
\begin{align*}
    \left\langle D_i^2 J_i( \omega_i; \boldsymbol{\omega}^{-i}), (v^i,\,v^i) \right\rangle \geqslant& \left(1-\delta_i +\frac{\delta_i}{N^2}\right)\int_0^T |v^i(t)|^2 \,dt \\
    &- C\int_0^T \left\{ |v^i(t)||y_i(t)| + |v^i(t)|^2 \sum_{j=1}^N |\varphi_{ji}(t)|  \right.\\
    &\left.\hspace{1.25cm}+ |v^i(t)|\sum_{j=1}^N |y_{ji}(t)| |\varphi_{ji}(t)|+ |v^i(t)|\sum_{j=1}^N |\psi_{ji}(t)| \right\}\,dt \\
    \geqslant& \left( 1- \delta_i + \frac{\delta_i}{N^2} - CT \right) \int_0^T |v^i(t)|^2 dt. 
\end{align*}
\normalsize
For small enough $T>0,$ the functional $\omega_i \mapsto J_i(\omega_i, \boldsymbol{\omega}_{-i}^*)$ is strictly convex, whence it follows that the critical point $\omega_i^*$ is a minimum. 
\end{proof}

Now, we extend local-in-time existence to its global-in-time analog.

\begin{theorem} \label{prop:Continuation}
Let us suppose that the assumptions of Theorem \ref{thm:LocalResult} hold. Given $T>0,$ there exist Nash equilibria in $\mathcal{A}_T.$
\end{theorem}
\begin{proof}
Let us take $\tau > 0$ sufficiently small, in accordance to Theorem \ref{thm:LocalResult}. That is, we take $\tau$ in such a way that, for every terminal time horizon $\tau^\prime$ less than or equal to $\tau,$ there exists a fixed point $\boldsymbol{\omega} \in C\left( \left[0,\tau^\prime\right] \right)^N$ of $\Phi$ which is a continuous Nash equilibrium. We let $\boldsymbol{\omega}^{1,\tau} \in C\left( \left[0,\tau\right] \right)^N$ be the Nash equilibrium on $\left[0,\tau\right]$ corresponding to the initial datum $x_{0,1},\ldots,x_{0,N},$ and we denote by $x^{1,\tau}$ the corresponding state. Then, we consider the Nash equilibrium $\boldsymbol{\omega}^{2,\tau}$ on $\left[0,\tau\right]$ corresponding to the initial datum $x^{1,\tau}_1(\tau),\ldots,x^{1,\tau}_N(\tau),$ denoting by $x^{2,\tau}$ its corresponding state. Proceeding inductively, we build the sequence $\left\{ \boldsymbol{\omega}^{j,\tau} \right\}_{j=1}^M,$ for large enough positive integers $M,$ with the following properties:
\begin{itemize}
    \item $\tau = T/M;$
    \item If $x^{j,\tau}$ is the state corresponding to $\boldsymbol{\omega}^{j,\tau}$ on $\left[0,\tau\right],$ then $\boldsymbol{\omega}^{j+1,\tau}$ is the continuous Nash equilibrium on this interval (the unique fixed point of $\Phi$ there) corresponding to the initial datum $x^{j,\tau}_1(\tau),\ldots,x^{j,\tau}_N(\tau).$
\end{itemize}
Let us denote by $\boldsymbol{\omega}^{*,\tau}$ the strategy satisfying
$$
\boldsymbol{\omega}^{*,\tau}(t) := \boldsymbol{\omega}^{j,\tau}\left( t - (j-1)\tau \right) \hspace{1.0cm} \left(j \in \left\{ 1,\ldots,M \right\},\, t \in \left[(j-1)\tau,j\tau\right[\right).
$$
It is clear that $\boldsymbol{\omega}^{*,\tau}|_{\left[0,T\right]} \in \Pi_{i=1}^N \mathcal{A}_i.$ We proceed to check that it is a Nash equilibrium for our problem. In effect, for each $s \in \left[0,T\right]$ and $\overline{x}_0 \in \left[-R,R\right]^d,$ let us denote by 
$$
\widetilde{J}_i^{\overline{x}_0,s}(\omega_i;\boldsymbol{\omega}_{-i}) := \int_0^s \left[\delta_i \left( \overline{\omega}(t) - x_i(t) \right)^2 + (1-\delta_i)\left( \omega_i(t) - x_i(t) \right)^2 + \lambda_i( \overline{x}(t)) \right]\,dt,
$$
where the superscript $\overline{x}_0$ in $J_i^{\overline{x}_0,s}$ means that $\overline{x}_0$ is the initial condition for the judgments $x.$ Moreover, for each $t \in \left[0, \tau \right],$ $i\in \mathcal{N},$ $j \in \left\{ 1,\ldots,M \right\}$ and $\omega_i \in \mathcal{A}_i,$ let us write:
$$
\begin{cases}
\widetilde{\omega}_i^{j}(t) = \omega_i(t+(j-1)\tau), \\
\widetilde{\omega}^{*,\tau,j}_i(t):= \omega^{*,\tau}_i(t+(j-1)\tau), \\
\overline{x}^{*,j}_0 = \left\{ x^{\tau}_k\left( (j-1)\tau \right) \right\}_{k=1}^N.
\end{cases}
$$
Under these notations, we derive
\begin{align*}
    J_i^{x_0,T}\left( \omega_i^{*,\tau};\boldsymbol{\omega}_{-i}^{*,\tau} \right) &= \sum_{j=1}^{M } \widetilde{J}_i^{\overline{x}^{*,j}_0,\tau}(\widetilde{\omega}_i^{*,j};\widetilde{\boldsymbol{\omega}}_{-i}^{j,*,\tau}) \\
    &\leqslant \sum_{j=1}^{M } \widetilde{J}_i^{\overline{x}^{*,j}_0,\tau}(\widetilde{\omega}_i^{j};\widetilde{\boldsymbol{\omega}}_{-i}^{j,*,\tau}) \\
    &= J_i^{x_0,T}\left( \omega_i; \boldsymbol{\omega}_{-i}^{*,\tau} \right).
\end{align*}
This proves that $\boldsymbol{\omega}^{*,\tau}$ is a Nash equilibrium.
\end{proof}

We notice that the strategy $\boldsymbol{\omega}^{*,\tau}$ we have built in the proof of Proposition \ref{prop:Continuation} is not necessarily continuous. In effect, even the choice of the continuation step $\tau$ is arbitrary, and is likely to lead to distinct competitive equilibria. However, we will argue that, upon letting $\tau \downarrow 0,$ there exists a limiting continuous profile. This is the content of our next result.

\begin{theorem} \label{thm:InfinitesimalContinuation}
Let us consider the assumptions of Theorem \ref{thm:LocalResult} as valid. Then, for each $T>0,$ there exist continuous Nash equilibria $\boldsymbol{\omega}^* \in \mathcal{A}_T.$ 
\end{theorem}
\begin{proof}
For the sake of clarity, we divide this proof in three steps. Firstly, we build a continuous approximation of the strategy $\boldsymbol{\omega}^*$ we constructed in the proof of Proposition \ref{prop:Continuation}. Secondly, we argue that the curve we defined is an approximate Nash equilibrium. Thirdly, we conclude the result via a compactness argument. 

Prior to step one, we fix the following notations. The set $\mathcal{T}_T$ comprises those $\tau > 0$ for which $ T / \tau$ is a large enough positive integer (say $T/\tau \geqslant M_0$) in such a way that, for each $\tau \in \mathcal{T}_T,$ the strategy $\boldsymbol{\omega}^{*,\tau}\in \mathcal{A}_T$ we constructed in the proof of Proposition \ref{prop:Continuation} is a Nash equilibrium. We are now ready to proceed with the present proof. 

\textbf{Step 1.} Construction of an approximate polygonal path.

Let us define $\boldsymbol{\gamma}^{\tau}$ as the polygonal path connecting the points $\boldsymbol{\omega}^{*,\tau}(j\tau),$ for $j\in\left\{0,...,M\right\}.$ We claim that $\left| \gamma^\tau - \boldsymbol{\omega}^{*,\tau} \right| = O(\tau).$ In effect, let us take $t \in \left[(j-1)\tau,j\tau\right[.$ Then, upon writing $s= t-(j-1)\tau,$ we have
\scriptsize
\begin{align*}
    | &\omega_i^{*,\tau}(t) - \omega_i^{*,\tau}((j-1)\tau)| = \left| \omega_i^{j,\tau}(s) - \omega_i^{j,\tau}(0) \right| \\
    =& \Bigg| P_{\mathcal{A}_i}\left( x_i^{j,\tau} - \frac{\delta_i}{(1-\delta_i)N}\left(\overline{\omega}^{j,\tau} - x^{j,\tau}_i \right) - \frac{1}{1-\delta_i}\sum_{l=1}^N K_{li}(\omega^{j,\tau}_l - x^{j,\tau}_l )\varphi^{j,\tau}_{l,i} \right)(s) \\
    &\,- \, P_{\mathcal{A}_i}\left( x^{j,\tau}_i - \frac{\delta_i}{(1-\delta_i)N}\left(\overline{\omega}^{j,\tau} - x^{j,\tau}_i \right) - \frac{1}{1-\delta_i}\sum_{l=1}^N K_{li}(\omega^{j,\tau}_l - x^{j,\tau}_l )\varphi^{j,\tau}_{l,i} \right)(0) \Bigg| \\
    \leqslant&\, C\left( \left| x^{j,\tau}_i(s) - x^{j,\tau}_i(0) \right| + \sum_{l=1}^N\left|K_{li}^\prime(\omega_l^{j,\tau}(s) - x_l^{j,\tau}(s))\varphi_{li}^{j,\tau}(s)  - K_{li}^\prime(\omega_l^{j,\tau}(0) - x_l^{j,\tau}(0))\varphi_{li}^{j,\tau}(0)\right|\right) \\
    &+ \max_{i}\left[ \frac{\delta_i}{(1-\delta_i)N} \right] \max_{i}\left| \omega^{j,\tau}_i(s) - \omega^{j,\tau}_i(0) \right| \\
    \leqslant&\, Cs + \max_{i}\left[ \frac{\delta_i}{(1-\delta_i)N} \right] \max_i\left| \omega_i^{*,\tau}(t) - \omega_i^{*,\tau}((j-1)\tau) \right|,
\end{align*}
\normalsize
whence
\begin{equation} \label{eq:PolygonalComparion1}
    \max_i | \omega_i^{*,\tau}(t) - \omega_i^{*,\tau}((j-1)\tau)|\leqslant C\tau \hspace{1.0cm} \left( t \in \left[(j-1)\tau, j\tau \right[ \right). 
\end{equation}
Next, for $t \in \left[(j-1)\tau,\,j\tau \right[,$ we estimate
\small
\begin{align} \label{eq:PolygonalComparison2Step1}
    \begin{split}
        | \omega_i^{*,\tau}(t) - \omega_i^{*,\tau}(j\tau)| \leqslant&\, |\omega_i^{*,\tau}(t) - \omega_i^{*,\tau}((j-1)\tau)| + | \omega_i^{*,\tau}((j-1)\tau)-\omega_i^{*,\tau}(j\tau-)|\\
        &\,+ |\omega_i^{*,\tau}(j\tau-) - \omega_i^{*,\tau}(j\tau)| \\
        \leqslant&\, C\tau + |\omega_i^{*,\tau}(j\tau-) - \omega_i^{*,\tau}(j\tau)|.
    \end{split}
\end{align}
\normalsize
Since $x^{j,\tau}_i(\tau) = x^{j+1,\tau}(0)$ and $\varphi_{li}^{j,\tau}(\tau) = 0,\,l\in\mathcal{N},$ we obtain
\scriptsize
\begin{align} \label{eq:PolygonalComparison2Step2}
    \begin{split}
        &|\omega_i^{*,\tau}(j\tau-) - \omega_i^{*,\tau}(j\tau)| = |\omega_i^{j,\tau}(\tau) - \omega_i^{j+1,\tau}(0)| \\
    &\leqslant\, \Bigg| P_{\mathcal{A}_i}\left( x_i^{j,\tau} - \frac{\delta_i\left(\overline{\omega}^{j,\tau} - x^{j,\tau}_i  \right)}{(1-\delta_i)N} - \frac{1}{1-\delta_i}\sum_{l=1}^N K_{li}(\omega^{j,\tau}_l - x^{j,\tau}_l )\varphi^{j,\tau}_{l,i} \right)(\tau) \\
    &\,\hspace{0.5cm}- \, P_{\mathcal{A}_i}\left( x^{j+1,\tau}_i - \frac{\delta_i}{(1-\delta_i)N}\left(\overline{\omega}^{j+1,\tau} - x^{j+1,\tau}_i \right) \right.\\
    &\left.\,\hspace{1.6cm}- \frac{1}{1-\delta_i}\sum_{l=1}^N K_{li}(\omega^{j+1,\tau}_l - x^{j+1,\tau}_l )\varphi^{j+1,\tau}_{l,i} \right)(0) \Bigg| \\
    &\leqslant \max_i \left[ \frac{\delta_i}{(1-\delta_i)N} \right] \max_i \left| \omega_i^{j,\tau}(\tau) - \omega_i^{j+1,\tau}(0) \right| \\
    &\hspace{0.5cm}+ \left| \frac{1}{1-\delta_i}\sum_{l=1}^N K_{li}(\omega^{j+1,\tau}_{li}(0) - x^{j+1,\tau}_{li}(0))\varphi^{j+1,\tau}_{li}(0) \right| \\
    &\leqslant \max_i \left[ \frac{\delta_i}{(1-\delta_i)N} \right] \max_i \left| \omega_i^{*,\tau}(j\tau -) - \omega_i^{j+1,\tau}(j\tau) \right| + C\tau.
    \end{split}
\end{align}
\normalsize
Putting \eqref{eq:PolygonalComparison2Step1} and \eqref{eq:PolygonalComparison2Step2} together, we derive
\begin{equation} \label{eq:PolygonalComparison2}
    \max_i \left| \omega_i^{*,\tau}(t) - \omega_i^{*,\tau}(j\tau) \right| \leqslant C\tau \hspace{1.0cm} \left( t \in \left[(j-1)\tau, j\tau \right[ \right).
\end{equation}
From estimates \eqref{eq:PolygonalComparion1} and \eqref{eq:PolygonalComparison2}, we deduce that
\begin{equation} \label{eq:Approximation_Polygonal}
    \left| \boldsymbol{\omega}^{*,\tau} - \gamma^\tau \right| = O(\tau).
\end{equation}

\textbf{Step 2.} The polygonal path $\gamma^\tau$ is an approximate Nash equilibrium.

From the stability estimate \eqref{prop:ContinuityOnOmega} and \eqref{eq:Approximation_Polygonal}, we obtain
$$
J_i\left( \omega^\tau_i;\boldsymbol{\omega}^\tau_{-i} \right) = J_i\left( \gamma^\tau_i;\boldsymbol{\gamma}^\tau_{-i} \right) + O(\tau),
$$
and
$$
J_i\left( \omega_i;\boldsymbol{\omega}^\tau_{-i} \right) = J_i\left( \omega_i;\boldsymbol{\gamma}^\tau_{-i} \right) + O(\tau).
$$
These two relations above easily imply that
\begin{equation} \label{eq:ApproxNE_Polygonal}
    J_i\left( \gamma^\tau_i;\boldsymbol{\gamma}^\tau_{-i} \right) \geqslant J_i\left( \omega_i;\boldsymbol{\gamma}^\tau_{-i} \right) - C\tau,
\end{equation}
for each admissible $\omega_i,$ from where the approximate Nash equilibrium property follows.

\textbf{Step 3.} Up to a subsequence, the approximating sequence $\left\{ \gamma^\tau \right\}_{\tau \in \mathcal{T}_T} = \left\{ \gamma^{\frac{T}{M}} \right\}_{M=M_0}^{\infty}$ has a limiting path $\boldsymbol{\omega}^* \in \mathcal{A}_T \cap C\left(\left[0,T\right]\right)^N.$ Furthermore, $\boldsymbol{\omega}^*$ is a Nash equilibrium.

Estimates \eqref{eq:PolygonalComparion1} and \eqref{eq:PolygonalComparison2} allow us to conclude that the slope of $\gamma^\tau$ on each interval $\left[(j-1)\tau,\,j\tau\right[$ is bounded by a universal constant $C,$ i.e., with $C$ being independent of the interval and of $\tau \in \mathcal{T}_T.$ In general, if $0 \leqslant  s \leqslant t \leqslant T,$ we take $j,k$ such that
$$
(j-1)\tau \leqslant s < j\tau \text{ and } (k-1)\tau \leqslant t \leqslant k\tau. 
$$
In this way, we infer
\begin{align} \label{eq:Equicontinuity}
  \begin{split}
    \left| \gamma^\tau(t) - \gamma^\tau(s) \right| \leqslant&\, \left| \gamma^\tau(t) - \gamma^\tau((k-1)\tau) \right| + \sum_{l=j}^{k-1} \left| \gamma^\tau(l\tau) - \gamma^\tau((l-1)\tau) \right| \\
    &+ \left| \gamma^\tau(t) - \gamma^\tau((j-1)\tau) \right| \\
    \leqslant&\, C(t-s).
  \end{split}
\end{align}
We remark that, if $j = k,$ then the summation in the right-hand side of \eqref{eq:Equicontinuity} vanishes. We conclude that the sequence $\left\{ \gamma^\tau \right\}_{\tau}$ is equicontinuous. Since $\left\{ \boldsymbol{\omega}^{*,\tau} \right\}_{\tau }$ is uniformly bounded (as $\left| \boldsymbol{\omega}^{*,\tau} \right|\leqslant R),$ Eq. \eqref{eq:Approximation_Polygonal} implies that $\left\{ \gamma^\tau \right\}_\tau$ also is. Therefore, by the Arzel\`a-Ascoli Theorem, there exists a subsequential uniform limit $\boldsymbol{\omega}^*,$ i.e.,
$$
\boldsymbol{\gamma}^{\tau_j} \rightarrow \boldsymbol{\omega}^*,
$$
uniformly on $\left[0,T\right],$ as $\tau_j \rightarrow 0$ within $\mathcal{T}_T.$ From \eqref{eq:ApproxNE_Polygonal}, we deduce that $\boldsymbol{\omega}^*$ is a  continuous Nash equilibrium. \end{proof}

We now provide a uniqueness result. This indicates that we should indeed concentrate on continuous equilibria.

\begin{theorem}
If we assume $\zeta_i \equiv 0$ and $K_{ii} \equiv 0,$ then there is a unique continuous Nash equilibrium.
\end{theorem}
\begin{proof}
Under the current suppositions, we see from \eqref{eq:FixedPoint} that a Nash equilibrium $\boldsymbol{\omega}^*,$ alongside its corresponding state $x^*,$ have the following property: given $0\leqslant t_0<t_1\leqslant T,$ the strategy $\boldsymbol{\omega}^*(\cdot + t_0)|_{\left[0,t_1-t_0\right]}$ is a Nash equilibrium corresponding to the initial data $x^*(t_0).$ Let us consider two Nash equilibria $\boldsymbol{\omega}^1,\boldsymbol{\omega}^2$ on $\left[0,T\right].$ For sufficiently small $\tau > 0,$ it follows from Theorem \ref{thm:LocalResult} that $\boldsymbol{\omega}^1|_{\left[0,\tau\right]} = \boldsymbol{\omega}^2|_{\left[0,\tau\right]}.$ Let $\overline{t}$ be the largest real number in $\left[0,T\right]$ for which this relation holds on $\left[0,\overline{t}\right].$ Denoting by $\overline{x}$ the state determined by $\boldsymbol{\omega}^1|_{\left[0,\overline{t}\right]} = \boldsymbol{\omega}^2|_{\left[0,\overline{t}\right]},$ a unique Nash equilibrium on $\left[\overline{t},\overline{t} + \tau \right]$ originates from the initial condition $\overline{x}(\overline{t}).$ However, $\boldsymbol{\omega}^1|_{\left[\overline{t},\overline{t}+\tau\ \right]}$ and $\boldsymbol{\omega}^2|_{\left[\overline{t},\overline{t}+\tau\ \right]}$ are two such possibilities, whence they must coincide there. We conclude that $\boldsymbol{\omega}^1 \equiv \boldsymbol{\omega}^2$ on $\left[0,T\right].$
\end{proof}

\subsection{Some results on the asymptotic behavior of the Nash equilibria}

Throughout the remainder of this section, let us fix for each $T>0$ a continuous Nash equilibrium $\boldsymbol{\omega}_T = \left( \omega_{1,T},\, \ldots,\, \omega_{N,T} \right)^\intercal.$ We denote the state variable corresponding to $\boldsymbol{\omega}_T$ by $\boldsymbol{x}_T=\left( x_{1,T},\,\ldots,\, x_{N,T} \right)^\intercal.$ In what follows, we present two results concerning the asymptotic behavior of $\boldsymbol{\omega}_T$ and $\boldsymbol{x}_T,$ as $T \rightarrow \infty.$

\begin{proposition}
Let us assume that there are $a,b \in \mathbb{R}^d$ such that:
\begin{itemize}
    \item For each $j \in \mathcal{N},$ the convergences $x_{j,T}(T) \rightarrow a \in A_j,$ as $T \rightarrow \infty,$ and $\omega_{j,T}(T) \rightarrow b \in A_j,$ as $T \rightarrow \infty,$ hold;
    \item For some $i \in \mathcal{N},$ the point $b$ belongs to the interior of the set $A_i.$
\end{itemize}
Then, $a=b.$
\end{proposition}
\begin{proof}
From \eqref{eq:FixedPoint}-\eqref{eq:DefnOfPhi}, the current assumptions imply
\begin{equation} \label{eq:ThmAsymp1}
    b = a - \frac{\delta_i}{(1-\delta_i)N}\left( b - a \right).
\end{equation}
Since $\delta_i < 1,$ Eq. \eqref{eq:ThmAsymp1} allows us to conclude that $a=b.$
\end{proof}

The above proposition tells us that, for large times, the true and expressed judgments cannot aggregate at two distinct consensus, as long as the true one is interior to at least one of the action spaces. We next investigate what we can say when each individual true and expressed judgments converge, as the terminal time horizon goes to infinity.

\begin{proposition} \label{prop:Asymptotics_when_WeqToX}
If $\omega_{i,T}(T),x_{i,T}(T) \xrightarrow{T \rightarrow \infty} a_i,$ then either $a_i \in \partial A_i,$ or else $a_i = \frac{1}{N}\sum_{j=1}^N a_j.$
\end{proposition}
\begin{proof}
Let us assume that, for some $i\in\mathcal{N},$ we verify $\omega_{i,T}(T),x_{i,T}(T) \rightarrow a_i,$ with $a_i$ in the interior of $A_i,$ as $T \rightarrow \infty.$ Then, by Eqs. \eqref{eq:FixedPoint}-\eqref{eq:DefnOfPhi}, we see that 
$$
a_i = a_i - \frac{\delta_i}{(1-\delta_i)N}\left( \frac{1}{N}\sum_{j=1}^N a_j - a_i \right),
$$
whence $a_i = \frac{1}{N}\sum_{j=1}^N a_j.$
\end{proof}

In many examples, we will see that $\boldsymbol{x}_T(T) - \boldsymbol{\omega}_T(T) \rightarrow 0,$ as $T \rightarrow \infty,$ but this is not always the case, as we will show by means of a counterexample (cf. Figure \ref{fig:PlotsBenchmarkInteresting}).

\section{Numerical experiments} \label{sec:Numerics}

Throughout the present section, we will provide several illustrations of situations our models capture, and draw some conclusions implied by them, which we formulate as ``Stylized Facts'' of our models. For the convenience of the reader, we recall in Table \ref{tab:modeling} the main modeling elements.

\begin{table}[!htp]
\centering
\begin{tabular}{@{}cr@{}}
\toprule
Symbol     & Meaning                                        \\ \midrule
$x_i$      & True judgement                                 \\
$\omega_i$ & Expressed judgement                            \\
$K_{ij}$        & Interaction kernels                             \\
$\delta_i$ & Persuasion ($\delta_i>0$)/conformity ($\delta_i<0$) parameter                \\
$\zeta_i$  & Sensitivity to exogenous objective information \\
$\lambda$  & Exogenous objective information parameter      \\ \bottomrule
\end{tabular}
\caption{Recalling the modeling elements.}
\label{tab:modeling}
\end{table}

Henceforth, we fix the kernel $K_{ij}(z) \equiv K(z) = a(z)z,$ where
\begin{equation} \label{eq:FnDefKernelExp}
    a(z) = 
\begin{cases}
c_{\alpha,\,R}\exp\left( \frac{\alpha}{|z|^2 - R^2} \right)\hspace{1.0cm} &\text{if } |z| < R , \\
0 &\text{otherwise,}
\end{cases}
\end{equation}
and $c_{\alpha,\,R} := e^{\frac{\alpha}{R^2}}$ is a normalizing constant, in such a way that $a(0) = 1.$ Moreover, we take $\alpha = 0.1$ and $R=0.5,$ unless we explicitly state otherwise. We now proceed to describe the numerical algorithms we will use to compute the optimal controls, as well as the associated states.

\subsection{A continuation fixed-point iterative numerical algorithm}

To compute the control and state variables we devised previously, we employ a two step approach. Firstly, for a sufficiently small continuation step, we provide an algorithm to solve the fixed point equation we described in Theorem \ref{thm:LocalResult} --- see Algorithm \ref{algo:fixedPoint}. Secondly, we propose in Algorithm \ref{algo:NE} a continuation method, where we continuously approximate the concatenated strategy by polygonal ones, thus obtaining an approximate control, as we showed in Theorem \ref{thm:InfinitesimalContinuation}. 

\begin{algorithm}[!htp]
\SetAlgoLined
\KwResult{Solution to the fixed point equation \eqref{eq:FixedPoint} and the corresponding state variable}
Initialize with a sufficiently small $T>0,$ the error variable $\epsilon,$ the initial data $x_0,$ and initial guess $(x^0,\omega^0,\varphi^0)$ for the state-control-adjoint tuple, the iteration variable $k=0,$ and the tolerance $\epsilon_0>0.$ \\
\While{$\epsilon \geqslant \epsilon_0$}{
1: Define $\varphi^{k+1}$ as the solution of \eqref{eq:AdjointParameters} with $x = x^{k}$ and $\omega = \omega^{k};$\\
2: Given $\omega^k,$ let $x^{k+1}$ be the solution of \eqref{eq:BasicModel} with $\omega = \omega^k$ and initial data $x_0;$\\
3: Let $\omega^{k+1} = \Phi\left[ \omega^k \right],$ where $x=x^{k+1}$ and $\varphi = \varphi^{k+1}$\\
4: Update $\epsilon$ according to some criterion, say, $\epsilon := \left\| x^{k+1} - x^k \right\|_\infty + \left\| \omega^{k+1} - \omega^k \right\|_\infty$;\\
5: Update $k \gets k+1.$
}
\Return{ $\omega^k,\, x^{k} $ }
\caption{Iterative algorithm}
\label{algo:fixedPoint}
\end{algorithm}

\begin{algorithm}[!htp]
\SetAlgoLined
\KwResult{Approximation to the optimal strategy and corresponding state variables}
Initialize with the time horizon $T>0$ (chosen as desired), let $\tau = T/M$ ($M$ a positive integer) small enough so that Algorithm 1 converges on $\left[0,\tau\right],$ the initial judgments profile $x_0,$ and a tolerance $\epsilon_0 > 0$ (for solving the fixed point equation at each step). \\
\For{$k \in \left\{0,\ldots,M\right\}$}{
1: Let $\left( \omega^k,x^k \right)$ be the control-state pair computed via Algorithm 1 on $\left[k\tau,(k+1)\tau\right]$ with initial data $x^{k}\left( k\tau \right) = x_{k} ;$\\
2: Set $x_{k+1} := x^k\left( (k+1)\tau \right) ;$\\
3: Update $k \gets k+1;$
}
4: Form the continuous approximation $\omega$ whose graph is the polygonal path connecting the points $\left\{ \left(k\tau,\, \omega^k\left(k\tau\right) \right) \right\}_{k=0}^{M}.$\\
5: Let $x$ be the solution of \eqref{eq:BasicModel} corresponding to $\omega.$\\
\Return{$\omega,\,x$}
\caption{Continuation algorithm}
\label{algo:NE}
\end{algorithm}

\subsection{Completely endogenous one-dimensional experiments}

Throughout this set of experiments, we carry out some numerical experiments that are completely endogenous, meaning that the dynamics of the movements of the agents' judgments stem solely from their interactions, i.e., $\zeta_i = 0,$ for each $i\in \mathcal{N}.$ We also resort to the one-dimensional setting ($d=1$). We summarize our results by formulating some ``stylized facts'' of our models, that is, some general statements of situations they capture for suitable configurations of parameters. We recall that the parameters $\delta_1,\,\ldots,\,\delta_N$ represent the persuasiveness/conformity of the agents. In each of the experiments, we will point out what are our choices for these, and we also present them in the legends of the corresponding plots.

We begin with the following:

\textbf{Stylized fact 1.} \textit{Strong opiners can group steadily (i.e., aggregate) in an extreme, even under the assumption of rational behavior.}

\textbf{Stylized fact 2.} \textit{Truthful centrists lead to a more numerous center, even with the presence of persuasive extremists.}

\textbf{Stylized fact 3.} \textit{If centrists are persuasive, then truthful radicals lead to sizable extreme groups.}

For two possible configurations of persuasiveness/conformity parameters, we obtain the outcomes we showcase in Figure \ref{fig:First2Exps1D}. Qualitatively, in the left panel of this figure, there are many more people aggregating around the center than in right one. We also observe in the latter a trapping phenomenon: as those that are in suspicion (i.e., mildly pending to one of the sides) radicalize, the states of the ones with stronger opinions are bound to be at least as extreme as former ones, cf. Lemma \ref{lem:Monotonicity}. Also, we notice that dissimulation, together with the ``physical'' bounds on the domain of admissible judgments, lead to radicalization, even under the assumption of full rationality of the players, cf.   \refcite{hegselmann2015opinion}. These two experiments indicate that the behavior of individuals in suspicion is key to the equilibrium outcome, whence to the first three stylized facts we considered so far.

% \begin{figure}[!htp]
% \centerline{\includegraphics[width=2.4in]{}}
% \vspace*{8pt}
% \caption{}
% \end{figure}

\begin{figure}[!htp]
    \centering
    \begin{subfigure}
        \centering
        \includegraphics[width=2.4in]{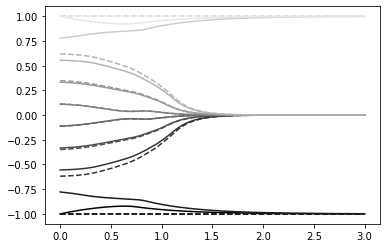}
    %     \caption{$\delta_i = |x_0|\wedge .8$}
    % \label{fig:1stExp}
    \end{subfigure}
    \begin{subfigure}
        \centering
        \includegraphics[width=2.4in]{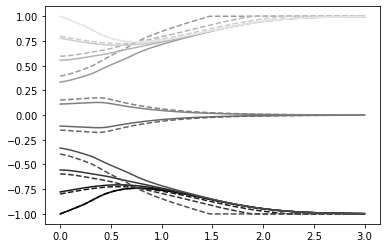}
        % \caption{$\delta_i = (1-|x_0|)\wedge .8$}
        % \label{fig:2ndExp}
    \end{subfigure}
    \caption{Two 1D experiments. In both plots, the horizontal axis represents the time variable, the vertical axis the judgments, and the initial true judgments are varying in an equispaced manner from $-1$ to $1.$ The solid lines are true judgments, whereas the dotted ones are the expressed counterparts (corresponding to the same color). We have put $N=10,$ $\delta_i = \left|x_{0,i}\right|\wedge .8$ on the left panel, and $\delta_i = (1-\left|x_{0,i}\right|)\wedge .8$ on the right one. In both experiments, we set $\zeta_i =0,$ for each $i \in \mathcal{N}.$}
    \label{fig:First2Exps1D}
\end{figure}

Next, we propose two more situations which our model encodes:

\textbf{Stylized fact 4.} \textit{If one of the extremes comprises conforming agents, whereas the opposite one has persuasive agents, with a uniform gradation in between, then we ought to see a prevalence of the persuasive side.}

\textbf{Stylized fact 5.} \textit{Sizable cohesive groups can be effective against radicalization.}

In effect, we present an asymmetric setting in Figure \ref{fig:ThirdExp1D}, in contradistinction to the ones we provided in Figure \ref{fig:First2Exps1D}. Players with initial (true) opinion closer to minus one are conforming, and those closer to one are persuasive. In between these two extremes, the tempers vary uniformly within the range $\left[-0.8,\,0.8\right].$ The fact that those strong opiners to the left of zero give in rapidly leads them to group together those in doubt (i.e., around the center), forming a strong cluster. This coalition does include some agents in suspicion to the right of zero (thus weakly persuasive). There are three players that are strong opining apropos of position ones and are irreducibly extremized. On the other hand, the fact that formed cluster becomes robust enough (i.e., sizable), together with the fact that they are in the range of interaction of the radicals (i.e., expressing opinions that are at a distance of less than a half away from theirs), they manage to deradicalize those three and attain a non-extremal consensus. The consensus, of course, leans strongly to the side of the persuasive extreme, which was expected to begin with, although it is insightful that this consensus did not simply turn out to be position one.

\begin{figure}[!htp]
    \centering
    \centerline{\includegraphics[width=2.4in]{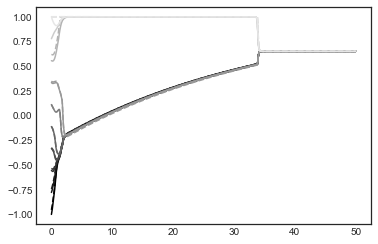}}
    \caption{In the current experiment, the horizontal axis represents the time variable, the vertical axis the distribution of judgments, and the initial true judgments are varying in an equispaced manner from $-1$ to $1.$ The solid lines are true judgments, whereas the dotted ones are the expressed counterparts (corresponding to the same color). We have put $N=10,$ $\delta_i = \left( x_{0,i} \wedge .8 \right) \vee (-.8)$ and $\zeta_i=0.$}
    \label{fig:ThirdExp1D}
\end{figure}

\subsection{Completely endogenous two-dimensional experiments}

We now consider two-dimensional experiments ($d=2$), but still without the influence of exogenous objective information ($\zeta_i =0$). In the first experiment, which we present in Figure \ref{fig:First2DExperiment}, we take $R = 2,$ keeping $\alpha = 0.1$ as before. We took $T=8$ in all plots, and we showcase the trajectories of judgements in them. The initial configuration of judgments is $\left\{\left(0.9,\,0.9 \right),\left(0.8,\,0.7 \right),\left(1,\,-0.35 \right),\left(-0.55,\,-0.35 \right)\right\}.$ This initialization mimics the position in the Economic-Social space of the candidates that participated in the US Presidential Election $2020,$ according to the Political Compass\footnote{\url{https://www.politicalcompass.org/uselection2020}}. The resulting configuration in Figure \ref{fig:First2DExperiment} agrees with the instinctive guess that, if we assume that the radius of interaction is large enough, in such a way that all players interact, then a consensus position arises. 

Now, in Figures \ref{fig:2ndSetOf2DExperiments} and \ref{fig:3rdSetOf2DExperiments}, when we consider the values of interaction radii $R \in \left\{ 1,\, 1.5 \right\},$ all else being the same, the two players in the top right aggregate at the top corner $(1,1),$ always selecting to express it as their opinion (i.e., they radicalize). In the first case, see Figure \ref{fig:2ndSetOf2DExperiments}, they do not interact with the one in the lower right, which then adheres to the lower right corner. However, in the second case, which we display in Figure \ref{fig:3rdSetOf2DExperiments}, this player ends up adhering to the top right as a result of the fact that her actual opinion begins to interact with its radical advocates (those expressing $(1,1)$ statically). In both experiments in Figure \ref{fig:2ndSetOf2DExperiments}, the agent in the lower left does not interact with anyone, resulting in her isolated radicalization in the $(-1,-1)$ position. 

\begin{figure}[!htp]
    \centering
    \begin{subfigure}
        \centering
        \includegraphics[width=2.4in]{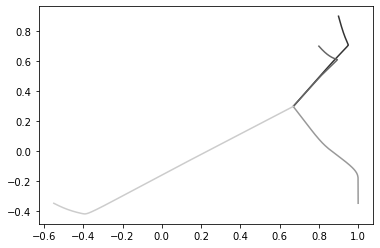}
        % \caption{Actual judgements.}
        % \label{fig:1stExp2D_actual}
    \end{subfigure}
    \begin{subfigure}
        \centering
        \includegraphics[width=2.4in]{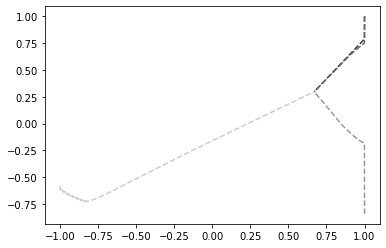}
        % \caption{Expressed judgements.}
        % \label{fig:1stExp2D_expressed}
    \end{subfigure}
    \begin{subfigure}
        \centering
        \includegraphics[width=2.4in]{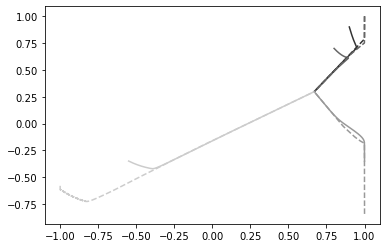}
        % \caption{Both state and control variables.}
        % \label{fig:1stExp2D_both}
    \end{subfigure}
    \caption{We present actual judgments (upper left panel), expressed ones (upper right panel), and both trajectories together (lower panel). The initial configuration of true judgments is $\boldsymbol{x}_0=\left(\left(0.9,\,0.9 \right),\left(0.8,\,0.7 \right),\left(1,\,-0.35 \right),\left(-0.55,\,-0.35 \right)\right)^\intercal.$ We also took $\delta_i=|\boldsymbol{x}_{0,i}|\wedge .8,\,\zeta_i = 0,\,T=8,$ and $R=2.$ Each horizontal (resp., vertical) axis corresponds to the values of the first (resp., second) dimension of the judgments.  The evolution of each judgment (true or expressed) departs from the corresponding initial condition in the direction of the terminal state, to which it converges --- this terminal state being the same for both judgments (here, a consensus). We observe the analogous behavior for the expressed ones, but slightly more extremized (specially in the beginning, cf. the lower panel), but which terminates at the same consensus as the true ones, cf. Proposition \ref{prop:Asymptotics_when_WeqToX}. }
    \label{fig:First2DExperiment}
\end{figure}

\begin{figure}[!htp]
    \centering
    \begin{subfigure}
        \centering
        \includegraphics[width=2.4in]{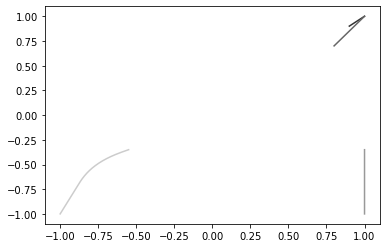}
        % \caption{$R=1$}
        % \label{fig:ReducedR1}
    \end{subfigure}
    \begin{subfigure}
        \centering
        \includegraphics[width=2.4in]{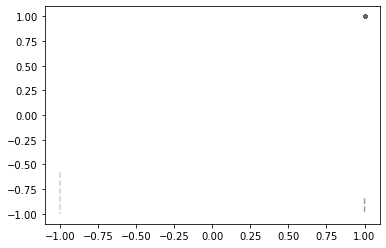}
        % \caption{$R=1.5$}
        % \label{fig:ReducedR1andAHalf}
    \end{subfigure}
    \caption{We present actual judgments (left panel) and the expressed ones (right panel). The initial configuration of true judgments is $\boldsymbol{x}_0=\left(\left(0.9,\,0.9 \right),\left(0.8,\,0.7 \right),\left(1,\,-0.35 \right),\left(-0.55,\,-0.35 \right)\right)^\intercal.$ We also took $\delta_i=|\boldsymbol{x}_{0,i}|\wedge .8,\,\zeta_i = 0,\,T=8,$ and $R=1.$ Each horizontal (resp., vertical) axis corresponds to the values of the first (resp., second) dimension of the judgments.  The evolution of each judgment (true or expressed) departs from the corresponding initial condition in the direction of the terminal state, to which it converges --- this terminal state being the same for both judgments. In the present experiment, the judgments (real \textit{and} expressed) of each player converge to the corner of $\left[0,1\right]^2$ which is closest to their initial true judgment. }
    % Effect of changing the radius on the competitive equilibria configurations. We put $R=1$ in the above experiment. On the left panel, we show the evolution of true opinions, whereas in the right one we present that of the expressed judgements.
    \label{fig:2ndSetOf2DExperiments}
\end{figure}

\begin{figure}[!htp]
    \centering
    \begin{subfigure}
        \centering
        \includegraphics[width=2.4in]{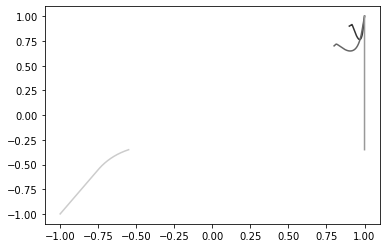}
        % \caption{$R=1$}
        % \label{fig:ReducedR1}
    \end{subfigure}
    \begin{subfigure}
        \centering
        \includegraphics[width=2.4in]{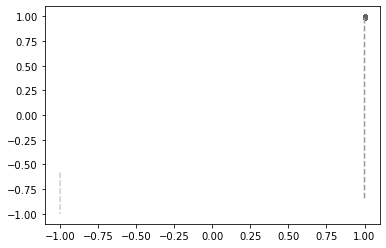}
        % \caption{$R=1.5$}
        % \label{fig:ReducedR1andAHalf}
    \end{subfigure}
    \caption{We present actual judgments (left panel) and the expressed ones (right panel). The initial configuration of true judgments is $\boldsymbol{x}_0=\left(\left(0.9,\,0.9 \right),\left(0.8,\,0.7 \right),\left(1,\,-0.35 \right),\left(-0.55,\,-0.35 \right)\right)^\intercal.$ We also took $\delta_i=|\boldsymbol{x}_{0,i}|\wedge .8,\,\zeta_i = 0,\,$ and $R=1.5.$ Each horizontal (resp., vertical) axis corresponds to the values of the first (resp., second) dimension of the judgments. The evolution of each judgment (true or expressed) departs from the corresponding initial condition in the direction of the terminal state, to which it converges --- this terminal state being the same for both judgments. In the present experiment, the judgments (real \textit{and} expressed) of each player but the one beginning in the lower right corner converge to the corner of $\left[0,1\right]^2$ which is closest to their initial true judgment. Both judgments of the remaining player converge to $(1,1).$ }
    \label{fig:3rdSetOf2DExperiments}
\end{figure}

\subsection{A one-dimensional experiment with exogenous influence}

For the remainder of this section, we assess the influence of external objective information in the model. We base our experiments in the description we made in Example \ref{ex:Example}. We consider a one-dimensional setting ($d=1$), and go back to using the parameters $\alpha = 0.1$ and $R=0.5.$ Firstly, we take $N=5$ agents, and our initial conditions are $(x_{0,1},\,\ldots,\, x_{0,5}) = (0,\,0.25,\,0.5,\,0.75,\,1).$ Now, the extremes correspond to positions labeled $0$ and $1.$ We interpret that all players want to minimize the expected value of the number of occurrences of an undesirable event, as in Example \ref{ex:Example}. As they act, they impact the intensity $\lambda$ of the nonhomogeneous Poisson point process which counts such manifestations. We assume that position $1$ is the ideal take for them to solve this issue, setting
$$
\lambda(\overline{x}) := \lambda_0 + \lambda_1\left(1 - \overline{x}\right).
$$
Unless we explicitly state otherwise, we fix $\lambda_1 = 1.$ Since the optimal strategies do not depend on $\lambda_0,$ we do not specify it here. Regarding their temper, we assume that advocates of extremes positions are symmetrically persuasive. The closer to the middle position of $0.5$ an agent is, the closer she is to being truthful. Explicitly, we take $\delta_i = \left( 2 | x_{0,i} - 0.5 | \right) \wedge 0.8.$ Now, as for their sensitivity to exogenous information, we set $\zeta_i = 100 x_{0,i},$ in such a way that players that are near position $1$ are more sensitive to this external objective information. As a player's initial actual opinion gets closer to zero, they tend to neglect this aspect of the problem and focus on the other elements of their performance criteria. This is in accordance to the phenomenon of confirmation bias, where people favor information that supports their prior positions. 

In the current setting, we propose to address the following stylized facts:

\textbf{Stylized fact 6.} \textit{Even under external object information, in realistic scenarios, rational dissimulating agents can group against the truth.}

\textbf{Stylized fact 7.} \textit{Even in face external objective information, an aggregate incorrect initial judgment can be persistent (if people are not truthful).}

We expose in Figure \ref{fig:1stExpObjInfo} the configuration that results from these elements. We observe a remarkable outcome: people acting strategically (i.e., envisaging to minimize their utilities) end up grouping majorly against the appropriate position. What happens is that players below $0.5$ slightly move towards it, whereas those above it move away from it. Moreover, the latter movement is more significant in magnitude than the former. Thus, the player in the center interacts more intensely with position $0$ players, resulting in her grouping together with them against $1.$ Just after time $3,$ when the agent initially in doubt seems to take a clear direction (the wrong one), the position $1$ advocates give up in trying to convince her, and aggregate at once at the correct spot.

The situation in Figure \ref{fig:1stExpObjInfo} is drastically distinct to the case in which no one is sensitive to external information (i.e., $\zeta_i = 0,$ for all $i\in \mathcal{N}$), \textit{ceteris paribus}, which we showcase in the left panel of Figure \ref{fig:PlotsBenchmarkInteresting} as a benchmark. The latter experiment is also insightful, as it shows that it is possible that players accommodate at an equilibrium in which actual and expressed judgments of some players (here, all but one) do not coincide asymptotically in time. It is an interesting experiment, showing the richness of dynamics we can obtain as competitive equilibria in our model.

\begin{figure}[!htp]
    \centering
    \begin{subfigure}
        \centering
        \includegraphics[width=2.4in]{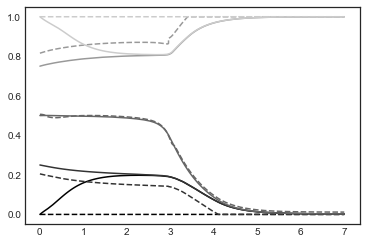}
    \end{subfigure}
    \caption{Experiment under the influence of exogenous objective information. The five initial true judgments being in an equispaced configuration from $0$ to $1,$ we take $\delta_i=\left(2|x_{0,i} - 0.5|\right)\wedge 0.8,$ $\zeta_i = 100 x_{0,i},$ and $\lambda(\overline{x})=\lambda_0 + 1  - \overline{x}$ (results are independent of $\lambda_0$). The horizontal axis represents time, whereas the vertical one is our scale of judgments. The continuous lines are true judgments, the dotted ones are the expressed counterpart (of the same color). }
    \label{fig:1stExpObjInfo}
\end{figure}

\begin{figure}[!htp]
    \centering
    \begin{subfigure}
        \centering
        \includegraphics[width=2.4in]{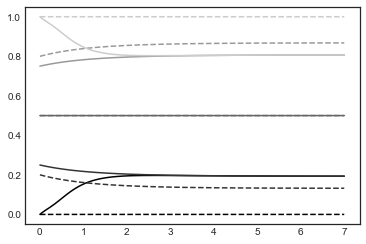}
        % \caption{$\zeta_i = 0,$ \textit{ceteris paribus}.}
        % \label{fig:BenchmarkCeterisParibus}
    \end{subfigure}
    \begin{subfigure}
        \centering
        \includegraphics[width=2.4in]{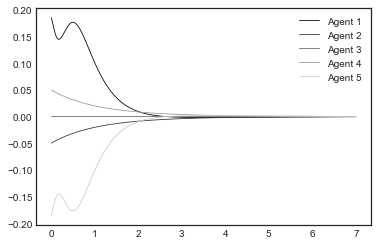}
        % \caption{$\dot{x}_i$}
        % \label{fig:DerivativesBenchmark}
    \end{subfigure}
    \caption{Benchmark for the experiment under the influence of exogenous objective information. The five initial true judgments being in an equispaced configuration from $0$ to $1,$ we take $\delta_i=\left(2|x_{0,i} - 0.5|\right)\wedge 0.8,$ $\zeta_i = 0.$ The horizontal axis represents time, whereas the vertical one is our scale of judgments. In the left panel, the continuous lines are true judgments, the dotted ones are the expressed counterpart (of the same color). In the right panel, we show $\dot{x}_i$ in the right one.}
    \label{fig:PlotsBenchmarkInteresting}
\end{figure}

If every player were truthful, i.e., $\delta_i = 0,$ for all $i \in \mathcal{N},$ \textit{ceteris paribus}, then we would obtain as the equilibrium the configuration we present in the left panel of Figure \ref{fig:ExpsObjTruthfuly}. In this case, we begin noticing that the expressed judgments differ from the actual ones only slightly, which is consistent with what we expect from truthful agents. Next, we see that in face of the exogenous information, all agents rapidly gather together in a state of doubt, i.e., around $0.5,$ not pending decidedly to neither side. Bundled together, they proceed to digest what they captured in extramental reality, walking gradually towards the correct side. Given enough time, they eventually reach the correct position, collectively finding the correct solution. We emphasize the remarkable fact that we maintained $\zeta_i = 100 x_{0i},$ whence our modeling of confirmation bias is still in force, highlighting the importance of truthfully in the resulting efficient collective behavior (from a social welfare perspective). We have raised the value of $\lambda_1$ to $7$ in the right panel of Figure \ref{fig:ExpsObjTruthfuly}, making the deviation between real and expressed judgments be a bit more significant, and players to reach an agreement sooner.

\begin{figure}[!htp]
    \centering
    \begin{subfigure} 
        \centering
        \includegraphics[width=2.4in]{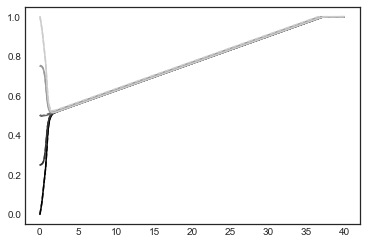}
        % \caption{Plot with $\delta_i = 0$ and $\lambda_1 = 1.$}
        % \label{fig:ExpObjHon1}
    \end{subfigure}
    \begin{subfigure} 
        \centering
        \includegraphics[width=2.4in]{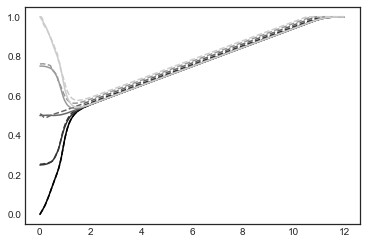}
        % \caption{Plot with $\delta_i = 0$ and $\lambda_1 = 1.$}
        % \label{fig:ExpObjHon1}
    \end{subfigure}
    \caption{Two experiments with objective information under complete truthfully, i.e., $\delta_i = 0,$ for every $i \in \mathcal{N}.$ In the left panel, we set $\lambda(\overline{x}) = \lambda_0 + 1-\overline{x},$ whereas $\lambda(\overline{x}) = \lambda_0 + 7(1-\overline{x})$ in the right one. For the remaining parameters, we took $N=5,$ the initial true judgments in an equispaced configuration from $0$ to $1,$ and $\zeta_i = 100 x_{0,i},$ in both plots. Furthermore, in the two figures, the horizontal axis represents time, whereas the vertical one comprises values of the judgments; continuous lines are real judgments, whereas the dashed ones (of the same color) are the expressed counterparts.}
    \label{fig:ExpsObjTruthfuly}
\end{figure}

We finish this section by showing in Figure \ref{fig:AlternativeInitConfigExtInfo} two situations in which the population has an initial average judgment in the opposite side of the correct one. There are seven individuals, having initial actual judgments $\left\{ 0,\,.1,\,.2,\,.3,\,.5,\,.7,\,.9\right\},$ with an average of about $.39 < .5.$ All the remaining parameters are as in Figure \ref{fig:1stExpObjInfo}. We compare the $\lambda_1=1$ case to the $\lambda_1=7$ one to highlight the persistence of the observed behavior relative to the strength of the exogenous information that players observe. Therefore, the initial average judgment of the population is nearer to position $0$ than to $1$ --- the aggregate is in suspicion of zero as being the best choice of action. Consequently, all but two of the players do quickly agree on assenting at a consensus consisting of the wrong position. The agents not complying with this position are the ones that begin closer to position $1,$ and they cannot convince the truthful one, initially in doubt.  

\begin{figure}[!htp]
    \centering
    \begin{subfigure}
        \centering
        \includegraphics[width=2.4in]{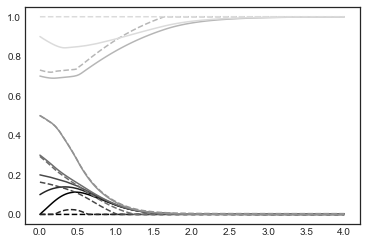}
        % \caption{$\lambda_1 = 1$}
    \end{subfigure}
    \begin{subfigure}
        \centering
        \includegraphics[width=2.4in]{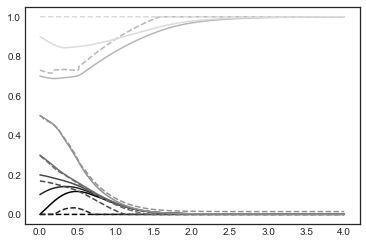}
        % \caption{$\lambda_1 = 7$}
    \end{subfigure}
    \caption{A case with a skewed initial set of judgments. We take $\boldsymbol{x}_0 = \left(0,\,.1,\,.2,\,.3,\,.5,\,.7,\,.9\right)^\intercal.$ Moreover, we put $\lambda(\overline{x}) = \lambda_0 + 1-\overline{x}$ (left panel) and $\lambda(\overline{x})=\lambda_0 + 7(1-\overline{x})$ (right panel). The remaining parameters are $\delta_i = \left(2\left|x_{0,i}-0.5\right|\right)\wedge 0.8$ and $\zeta_i=100 x_{0,i},$ for $i\in\mathcal{N}.$ In both plots, the horizontal axis represents time, the vertical one the judgments. The solid lines are real judgments, whereas the dotted ones are the expressed corresponding to the solid lines of the same color.}
    \label{fig:AlternativeInitConfigExtInfo}
\end{figure}

\section{The cooperative game} \label{sec:Coalitional}

In this section, we investigate the situation in which the populations' judgments still evolve under \eqref{eq:BasicModel}, and use the criteria $J_i$ we defined in \eqref{eq:ObjCriteria}, but now we do not assume that they are competing. Rather, we suppose that they lean towards building a consensus, at least in principle. In particular, from here on we will abandon the concept of Nash equilibrium, and look for an alternative notion which most appropriately represents the collective behavior of agents among a cooperative population. In this direction, we adopt a definition inspired in the one that the mathematical economist Vilfredo Pareto proposed\footnote{His actual words were (see   \refcite{pareto1964cours}, page 18): ``Nous étudierons spécialement l'\textit{équilibre économique}. Un système économique sera dit en équilibre si le changement d'une des conditions de ce système entraîne d'autres changements qui produiraient une action exactement opposée.''}: a social optimum is a set of strategies such that any individual improvement is necessarily detrimental to someone else. In practice, not every Pareto optimal strategy realistically expresses something that we would regard as a social optimum. However, we commonly identify a whole front of Pareto optimal controls, and it is a fair guess to look to such an optimizer within this set. We engage in this discussion more technically in the sequel.

\subsection{Necessary and sufficient conditions for a cooperative equilibrium}

\begin{definition}
A strategy $\boldsymbol{\omega}^*$ is a Pareto equilibrium if there does not exist $\boldsymbol{\omega}$ such that
$$
J_i(\boldsymbol{\omega}) \leqslant J_i(\boldsymbol{\omega}^*),\, i\in\mathcal{N},
$$
with the inequality holding strictly for at least one such $i.$
\end{definition}

The following result is classical in the optimization literature, cf. \refcite{leitmann1974cooperative} and  \refcite{ciarlet1989introduction}.

\begin{proposition} \label{prop:CharactPareto}
If the strategy $\boldsymbol{\omega}^*$ minimizes the functional $\sum_{i=1}^N \theta_i J_i,$ for some $\theta_1,...,\theta_N > 0$ subject to $\sum_{i=1}^N \theta_i = 1,$ then it is a Pareto equilibrium. Conversely, if $\boldsymbol{\omega}^*$ is a Pareto equilibrium and the functionals $\left\{ J_i \right\}_{i=1}^N$ are convex, then $\boldsymbol{\omega}^*$ must minimize $J = \sum_{i=1}^N \theta_i J_i,$ for some $\theta_1,\ldots,\theta_N \geqslant 0,$ such that $\sum_{i=1}^N \theta_i = 1.$
\end{proposition}

Proposition \ref{prop:CharactPareto} sheds light on the very definition of our performance criteria. In effect, from \eqref{eq:ObjCriteria_rewritten}, we see that a minimizer of the term $\widetilde{J}_i$ in $J_i$ is a Pareto optimal strategy, in a certain sense. If we divide $J_i$ by $1/2+\zeta_i,$ then we can even see a minimizer of $J_i$ as a Pareto equilibrium for the criteria $\left\{ \widetilde{J}_i,\, I \right\}.$ In this way, we can interpret the (partial) minimization of $J_i$ as if person $i$ internally tried to reconcile her possibly conflicting personal goals by accommodating in a suitable Pareto optimal strategy. 

We can interpret the parameters $\theta_1,\ldots, \theta_N$ when forming $J = \sum_{i=1}^N \theta_i J_i$ as the influence of the corresponding agent in the cooperative formation. In fact, $J$ is simply a weighted average of the set of criteria $\left\{ J_i \right\},$ the $\theta_1,\ldots,\theta_N$ being precisely the weights. Proposition \ref{prop:CharactPareto} shows that some Pareto optimal strategies are found precisely as minima of these $J.$ A consequence of this observation is that, through the choice of these weights, we can introduce a hierarchy in the model: a player having higher $\theta_i$ (relative to her peers) is such that her particular $J_i$ out-stands in the average, wherefrom we expect a possible consensus to be reached near to her judgment. 

Let us also remark that choosing $\theta_i = 0, $ for some $i,$ when forming the functional $J := \sum_{i=1}^N \theta_i J_i$ is not likely to be a realistic social optimum. Indeed, in this case, the functional $J_i$ of agent $i\in \mathcal{N}$ is disregarded in the formation of $J.$ In particular, if $\theta_i = 0,$ for all but one $i_0 \in \mathcal{N},$ then $J = J_{i_0},$ and a Pareto optimizer is simply the strategy which minimizes $J_{i_0}.$ This means that everyone would just act in such a way as to help the performance of player $i_0,$ which will most likely lead to an uninteresting behavior. For this reason, we will focus on the elements of the Pareto front corresponding to $\theta_1,\ldots,\theta_N > 0.$ 

As in the competitive setting, we split the proof in two parts (Theorems ~\ref{thm:ParetoNecessary} and ~\ref{thm:ParetoContinuousExistence} below) --- obtaining the necessary condition and then verifying their sufficiency for small time, which is then extended to all time.

\begin{theorem} \label{thm:ParetoNecessary}
Let us consider a Pareto equilibrium $\boldsymbol{\omega}^* = \left( \omega_1^{*},\ldots,\omega_N^{*}\right)^\intercal$ minimizing $\sum_{i=1}^N \theta_i J_i,$ for some $\theta_1,\ldots, \theta_N > 0$ with $\sum_{i=1}^N \theta_i = 1.$ Then, 
\begin{equation} \label{eq:NecessaryCondnsPareto}
    \boldsymbol{\omega}^*_i(t) = \Psi\left[ \boldsymbol{\omega}^* \right]_i(t),
\end{equation}
for the mapping $\Psi : \mathcal{A} \rightarrow \mathcal{A}$ whose $i-$th component is\footnote{For the definition of $P_{\mathcal{A}_i},$ see Remark \ref{rem:projection}.}
\small
\begin{align} \label{eq:Psi_defn}
  \begin{split}
    \Psi\left[ \boldsymbol{\omega} \right]_i := P_{\mathcal{A}_i}&\left( x_i - \frac{1}{\theta_i(1-\delta_i)}\left( \frac{1}{N}\sum_{j=1}^N \theta_ j \delta_j \left( \overline{\omega} - x_j \right) \right.\right. \\
    &\hspace{0.5cm}\left.\left.+ \sum_{j=1}^N K_{ji}^\prime(\omega_i - x_j)^\intercal \varphi_{j} \right) \right),
  \end{split}
\end{align}
\normalsize
where we define $\boldsymbol{x} = (x_1,\ldots,x_N)^\intercal $ as the state corresponding to $\boldsymbol{\omega},$ $\overline{\omega} := \frac{1}{N}\sum_{j=1}^N \omega_j,$ $\overline{x} := \frac{1}{N}\sum_{j=1}^N x_j,$ and the functions $\left\{\varphi_{j} \right\}_j$ solve
\small
$$
\begin{cases}
        -\dot{\varphi}_{j}(t) = - \sum_{l=1}^N K_{jl}^\prime\left(  \omega_l(t) - x_j(t)\right)^\intercal\varphi_{j}(t) + \theta_j \left\{ x_j(t) -  \left[ \left(1-\delta_j\right)\omega_j(t) + \delta_j\overline{\omega}(t) \right] \right\} \\
        \hspace{1.2cm}+ \frac{1}{N}\sum_{k=1}^N \theta_k \zeta_k \lambda^\prime\left(\overline{x}(t) \right), \, 0 < t < T,\\
        \varphi_{j}(T) = 0.
\end{cases}
$$
\normalsize
\end{theorem}
\begin{proof}
We proceed as in Theorem \ref{thm:NecessaryCondnNE} to deduce the necessary conditions \eqref{eq:NecessaryCondnsPareto} as a consequence of the first-order optimality conditions that the minimization of $\boldsymbol{\omega} \in \mathcal{A} \mapsto J\left( \boldsymbol{\omega} \right) \in \mathbb{R}$ requires. 
\end{proof}

As a necessary condition, we obtained Eq. \eqref{eq:NecessaryCondnsPareto}, which any minimizer of $\sum_{i=1}^N \theta_i J_i$ must satisfy --- a similar relation to Eq. \eqref{eq:FixedPoint}, which in turn must hold for Nash equilibria in the competitive setting. It is insightful to emphasize some differences between the two. Firstly, we notice that in \eqref{eq:NecessaryCondnsPareto} the adjustment in the expressed judgment of each player, relative to their actual one, depend on the individual only in their intensity, i.e., though the parameter $\frac{1}{\theta_i(1-\delta_i)}$ and the interaction variation strength $K^\prime_{ji}.$ This is radically distinct to what we observe in \eqref{eq:FixedPoint}. In the present cooperative setting, agents move upon considering an aggregate weighted deviation from the average overall expressed opinion (apart from the adjoint parameters), whereas in the competitive framework, each player moved based on her signal only. Moreover, in the cooperative framework, the adjoint parameters are homogeneous throughout the population --- this is not the case in the competitive counterpart. Finally, let us point out that in the current cooperative setting, Theorem \ref{thm:ParetoNecessary} does not necessarily constrain all the Pareto equilibria, but only those that minimize some convex combination $\sum_{i=1}^N \theta_i J_i.$ We will in fact focus on a finer subclass of such equilibria, as we proceed to discuss.

\begin{theorem} \label{thm:ParetoContinuousExistence}
(a) If $\theta_1,\ldots,\theta_N > 0,\, \sum_{i=1}^N\theta_i =1,$  
\begin{equation} \label{eq:ConditionsForPareto}
    \frac{1}{N}\sum_{j=1}^N \theta_j\delta_j  <  \min_{i \in \mathcal{N}}\left\{ \theta_i(1-\delta_i) \right\},
\end{equation}
and $T>0$ is small enough, then \eqref{eq:NecessaryCondnsPareto} admits a unique continuous solution $\boldsymbol{\omega} \in \mathcal{A}.$ Moreover, taking $T$ smaller, if necessary, this strategy becomes a Pareto equilibrium.

(b) For every $T>0,$ condition \eqref{eq:ConditionsForPareto} implies that there exists a continuous strategy $\boldsymbol{\omega} \in \mathcal{A}$ minimizing $J := \sum_{i=1}^N \theta_i J_i.$
\end{theorem}
\begin{proof}
$(a)$ We can show the existence of a continuous $\boldsymbol{\omega}$ just as we did in Theorem \ref{thm:LocalResult}, viz., by using the stability results we developed in Section \ref{sec:model} and Banach fixed point Theorem. The fact that $\boldsymbol{\omega}$ is a Pareto equilibrium, for sufficiently small $T,$ follows from the second-order conditions, just as in the end of the proof of Theorem \ref{thm:LocalResult}, but considering the joint dependence on the controls of the aggregate cost function --- we omit the details here.

$(b)$ We can fix a sufficiently small $\tau = T/M$ (with $M$ being a sufficiently large positive integer), and concatenate minimizers of the pieces of $J,$ in a similar way as we did in Proposition \ref{prop:Continuation}, thus forming a (possibly discontinuous) minimizer $\widetilde{\boldsymbol{\omega}}^\tau.$ Then, we can construct an approximate minimizer for $J$ by taking an appropriate polygonal $\gamma^\tau$ connecting the points $\left\{ \widetilde{\boldsymbol{\omega}}^\tau(j\tau) \right\}_{j=0}^M.$ By letting $\tau \downarrow 0$ through a suitable subsequence, we argue as in the proof of Theorem \ref{thm:InfinitesimalContinuation} to conclude that $\left\{ \gamma^\tau \right\}_{\tau}$ converges to a continuous minimizer $\boldsymbol{\omega}$ of $J.$
\end{proof}

We remark that, in Theorem \ref{thm:ParetoContinuousExistence}, we identify a subset of the whole Pareto front. We also do not affirm that the strategy $\boldsymbol{\omega}$ we identified in item $(b)$ is the unique minimizer of $J = \sum_{i=1}^N \theta_i J_i$ --- it is indeed \textit{a} (global) minimizer, and, by virtue of Proposition \ref{prop:CharactPareto}, a Pareto equilibrium. However, we argue that these already comprise a rich set and the restrictions we impose are not too strong. Indeed, item $(a)$ constrains, to an extent, how persuasive agents can be, as well as, from below, the influence that each individual has. For studying cooperative games, we advocate that these assumptions are reasonable. Aside from the fact that they make sense when we consider that agents cooperate, we further back up our claim by showing, through some numerical experiments, that we do obtain equilibria configurations representing realistic scenarios in this context.

Before proceeding to the numerical illustrations, we provide the counterpart of Proposition \ref{prop:Asymptotics_when_WeqToX} in the current cooperative setting. The proof is similar to it, whence we omit it here.

\begin{proposition} \label{prop:AsymptoticsForPareto_WeqToX}
Let us suppose that, for each $T>0,$ $\boldsymbol{\omega}(\cdot,\,T)$ is a Pareto optimal strategy minimizing $J = \sum_{i=1}^N \theta_i J_i,$ where $\theta_1,\,\ldots,\,\theta_N>0$ and $\sum_{i=1}^N \theta_i = 1.$ If $\omega_i(T,T),x_i(T,T) \rightarrow a_i,$ then either $a_i \in \partial A_i,$ for all $i \in \mathcal{N},$ or else
$$
\sum_{i=1}^N \theta_i \delta_i\sum_{i=1}^N a_i = \sum_{i=1}^N \theta_ i \delta_i  a_i .
$$
\end{proposition}

We notice that we cannot necessarily say, under the assumptions of Proposition \ref{prop:AsymptoticsForPareto_WeqToX}, that there is no interior clusterization. In effect, we will show by means of an example that such a phenomenon can happen in this context.

\subsection{Numerical experiments}

Let us recall the mapping $\Psi : \mathcal{A} \rightarrow \mathcal{A}$ we defined in \eqref{eq:DefnOfPhi}, that is,
\scriptsize
$$
\Psi[\boldsymbol{\omega}^*] := P_{\mathcal{A}_i}\left( x^*_i - \frac{1}{\theta_i(1-\delta_i)}\left( \frac{1}{N}\sum_{j=1}^N \theta_j \delta_j \left( \overline{\omega}^* - x_j^* \right) + \sum_{j=1}^N K_{ji}^\prime(\omega_i^* - x_j^* )^\intercal \varphi_{j}^* \right) \right),
$$
\normalsize
where $\left\{ \varphi_j^* \right\}_{j=1}^N$ are as in Theorem \ref{thm:ParetoNecessary} (with $\omega=\omega^*$). We adapt Algorithm \ref{algo:fixedPoint} to the present setting by replacing $\Phi $ by $\Psi$ in it. Subsequently, we compute the Pareto equilibria we show below via an adaptation of Algorithm \ref{algo:NE}, i.e., in which we use the modified Algorithm \ref{algo:fixedPoint}. We fix $K(z) = a(z)z,$ with $a$ as in Section \ref{sec:Competitive}, and we fix the parameters $\alpha = 0.1$ and $R=0.5.$

We proceed to begin the formulation of the first stylized fact of the current setting. Let us consider a cooperative group where one extreme is much more influential than the other. Then, the more extreme advocates of the less influential side might lead the movement towards consensus building. In a way, they will give in their more incisive opinion, so as to ``give the example'' for those that were on their side, but less radically, to follow them. If we assume that agents are persuasive, with a persuasion parameter increasing with respect to the opinion's size, then the more extreme opiners in the less influential side will want to be on the ``winning'' side. In the current cooperative framework, this amounts to the less influential extremists to switch sides more easily --- in a sense, their judgements are more flexible, since everyone is more concerned to building a consensus. However, those initially in suspicion of the less influential side can turn out more stubborn, since they care less about persuading.

\textbf{Stylized fact 8.} \textit{Persuasive agents with an initially less influential extreme opinion have a key role in cooperative formation --- in a way, they lead by example as they give in their radical position.}

\textbf{Stylized fact 9.} \textit{People with a more moderate opinion leaning to the less influential side of a binary proposition can be more stubborn in a cooperative formation. These are, to an extent, responsible for holding up the terminal consensus of being too radical on the initially more influential side.}

In our first experiment of this section, we provide an example illustrating how our model captures the two last stylized facts we stated, see Figure \ref{fig:1D_Pareto_Exp1}. We propose a setting where: (i) The influence $\theta_i$ of player $i$ on the cooperative group is higher the closer her initial judgment is to one; (ii) The persuasion level of an agent is proportional to her initial judgment's absolute value. We notice that the order of the expressed judgments eventually partially flips, with the players that were once nearer to position negative one becoming more intense advocates of the side where the cooperative group forms the consensus. The original agents in suspicion of position minus one show a bit more stubbornness, or some kind of persistence on this side --- here, this is due to the fact that they are more influential. It is also worthwhile to remark that this movement of the opiners of the negative side end up attracting a player slightly in suspicion of position one; thus, we observe the formation of two transitory clusters, the terminal consensus being met halfway between a state of doubt and position one. The latter discussion can also indicate how our model captures, as a result of social interactions, the phenomenon of individuals behaving in an edgy way: some who were once in an extreme side, suddenly become supporters of an opposite viewpoint. We gather some of these insights in the sequel.  

\begin{figure}[pb]
    \centering
    \includegraphics[width = 2.4in]{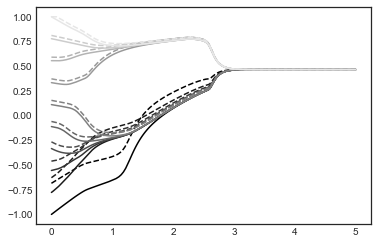}
    \caption{A one-dimensional Pareto equilibrium with $N = 10,$ $\theta_i = 2(i+1)/[N(N+1)],$ $\delta_i = |x_{i0}|/5,$ $\zeta_i=0.$ Initial judgments are distributed in an equispaced manner from $-1$ to $1.$ The horizontal axis consists of time, and the vertical one represents judgments. Solid lines are true judgments, whereas dotted ones are expressed judgments. }
    \label{fig:1D_Pareto_Exp1}
\end{figure}

The second phenomena we pay attention to concerns the role of symmetry in consensus building within a cooperative group. Namely, in realistic settings, it can lead to unsolvable disputes. In real-world situations, these excess of equality of several aspects among a population's individuals can possibly yield conflicts or other kinds of issues.

\textbf{Stylized fact 10.} \textit{From a social viewpoint, excess of temper and influence symmetry, among agents of a population, around a central state, might be an issue for agreement on a consensus.}

To illustrate how our model can capture this stylized fact, we propose the next experiment. Our framework is akin to that of   \refcite{rusinowska2019opinion}. Namely, we consider a setting with two groups of agents: one formed by three major players (leaders), the other constituted by four minor ones (followers). Among the major players, two of them are extreme opiners initially lying in opposite extremes of the admissible spectrum of judgments. The remaining major agent is a centrist/doubter. The four minor players initial judgments $\pm 1/3$ and $\pm 2/3.$ We clarify that we add the hierarchy here by stipulating that the influence of the leaders is higher than that of the followers. We showcase it in Figure \ref{fig:3L4F_sym}, where we consider a symmetric scenario. There, the system attains a terminal configuration consisting of interior clusterization --- as we had already announced that this could happen, in the discussion following Proposition \ref{prop:AsymptoticsForPareto_WeqToX}, and through this example we establish this claim. From this insightful illustration, we obtain a suggestion that, from a social perspective, excess of symmetry (relative to the agents' tempers and overall influence over the population) around the state of doubt, or the center, can be an issue for a cooperative group to agree on an asymptotic consensus --- people might arrive at an unsolvable dispute, in accordance to our previous discussion.

\begin{figure}[!htbp]
    \centering
    \begin{subfigure}
        \centering
        \includegraphics[width=2.4in]{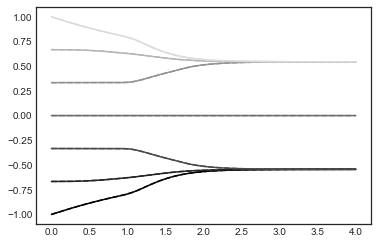}
    \end{subfigure}
    \caption{A configuration in which the extreme opiners and the doubter are hierarchically above, i.e., they have a higher influence on the cooperative formation. We fix $N=7,$ $(\theta_1,\ldots,\theta_7) = (5,1,1,5,1,1,5)/19,$ $\delta_i=|x_{0,i}|/5,$ $\zeta_i=0,$ and initial judgments are distributed in an equispaced manner from $-1$ to $1.$ The horizontal axis comprises time, and the vertical one the judgments. Solid lines are true judgments, and dotted ones are the expressed counterparts (corresponding to the same color).}
    \label{fig:3L4F_sym}
\end{figure}

We now analyze in Figure \ref{fig:3L4F_asym} an asymmetric situation: the followers (i.e., minor agents) on the positive side of the spectrum of opinions are slightly more influential than the other ones. We see the formation of two transitory clusters, in such a way that the agent in doubt bundles together those with a negative opinion, whereas the three with an opinion closer to one form another. The fact that the former cluster turns out more numerous provides them enough strength to make the final consensus not equal to the extreme position one. We also remark that, when the transitory clusters are formed, the larger one is more conforming --- expressing an overall opinion different than their actual ones --- whereas the smaller one is more persuasive, even radicalizing for a while. 

\begin{figure}[!htbp]
    \centering
    \begin{subfigure}
        \centering
        \includegraphics[width=2.4in]{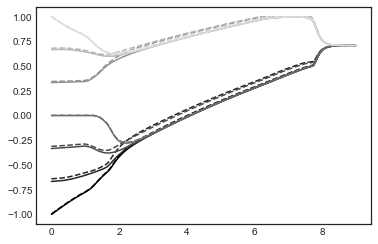}
    \end{subfigure}
    \caption{An alternative configurations where the extreme opiners and the doubter are still hierarchically superior, i.e., they have a higher influence on the cooperative formation. We now fix $N=7,$ $(\theta_1,\ldots,\theta_7) = (5,0.7,0.7,5,1.3,1.3,5)/19,$ $\delta_i=|x_{0,i}|/5,$ $\zeta_i=0,$ and initial judgments are distributed in an equispaced manner from $-1$ to $1.$ The horizontal axis comprises time, and the vertical one the judgments.	 Solid lines are true judgments, and dotted ones are the expressed counterparts (corresponding to the same color).}
    \label{fig:3L4F_asym}
\end{figure}

\section{Conclusions} \label{sec:Conclusions}

We proposed a model of social dynamics in which agents among a finite population interact through their stated judgments. Thus, each agent chooses which judgment she will express, and her true judgment is updated in accordance to those expressed by her peers. We worked in a control-theoretic framework, stipulating that the players updated their judgments by minimizing suitable performance criteria. The elements we considered for the design of these criteria were the agents' temper (persuasive, truthful, or conforming), as well as their sensitivity to objective exogenous information. We modeled the latter aspect as an average of the number of undesirable occurrences whose intensity was influenced by the agents' true judgments.

We first investigated the non-cooperative framework. We considered Nash equilibria (NE), proving a local-in-time existence and uniqueness result. Then, we showed by an iterative method that we could always find square integrable NE, but there was a degree of ambiguity --- the step size. We ruled out this issue by arguing via compactness that, as the continuation step goes to zero, the corresponding strategies we build converge to a continuous NE. In a particular case, we proved that there is in fact at most one (hence, a single one) such equilibrium. Using the richness of the equilibria we obtain by varying the model parameters, we then explored the implications of our model. We proposed a series of stylized facts to demonstrate how our results can provide conceptual insights about real world phenomena.   

Then, we proceeded to study the problem of cooperative formation through the light of our model. In this setting, we worked under the notion of Pareto equilibria. Although we did not provide a general identification of the Pareto front, we were able to find a rich set of Pareto equilibria. Our key assumptions were that neither agents were too persuasive, nor that there were agents with a too small overall influence over the population --- hypotheses that sound reasonable, from the viewpoint of cooperative games. The techniques we used to prove the technical results for this cooperative setting were quite similar to those we employed in the competitive one. We finished this part of the work by providing some numerical experiments, also in a similar form as in the previous (competitive) setting.

\section*{Acknowledgments}

YS, MOS and YT were financed in part by Coordena\c{c}\~ao de Aperfei\c{c}oamento de Pessoal de N\'ivel Superior - Brasil (CAPES) - Finance code 001. MOS was also partially financed by CNPq (grant \# 310293/2018-9) and by FAPERJ (grant \# E-26/210.440/2019).

\bibliographystyle{plain}
\bibliography{refs}

\begin{thebibliography}{10}

\bibitem{arieli2019multidimensional}
Itai Arieli and Manuel Mueller-Frank.
\newblock Multidimensional social learning.
\newblock {\em The Review of Economic Studies}, 86(3):913--940, 2019.

\bibitem{asch1955opinions}
Solomon~E Asch.
\newblock Opinions and social pressure.
\newblock {\em Scientific American}, 193(5):31--35, 1955.

\bibitem{bailo2018pedestrian}
Rafael Bailo, Jos{\'e}~A Carrillo, and Pierre Degond.
\newblock Pedestrian models based on rational behaviour.
\newblock In {\em Crowd Dynamics, Volume 1}, pages 259--292. Springer, 2018.

\bibitem{bala2001conformism}
Venkatesh Bala and Sanjeev Goyal.
\newblock Conformism and diversity under social learning.
\newblock {\em Economic Theory}, 17(1):101--120, 2001.

\bibitem{banerjee2019naive}
Abhijit Banerjee, Emily Breza, Arun~G Chandrasekhar, and Markus Mobius.
\newblock Naive learning with uninformed agents.
\newblock Technical report, National Bureau of Economic Research, 2019.

\bibitem{bacsar1998dynamic}
Tamer Ba{\c{s}}ar and Geert~Jan Olsder.
\newblock {\em Dynamic noncooperative game theory}.
\newblock SIAM, 1998.

\bibitem{basar2018handbook}
Tamer Basar and Georges Zaccour.
\newblock {\em Handbook of dynamic game theory}.
\newblock Springer, 2018.

\bibitem{bauso2016opinion}
Dario Bauso, Hamidou Tembine, and Tamer Basar.
\newblock Opinion dynamics in social networks through mean-field games.
\newblock {\em SIAM Journal on Control and Optimization}, 54(6):3225--3257,
  2016.

\bibitem{blondel2010continuous}
Vincent~D Blondel, Julien~M Hendrickx, and John~N Tsitsiklis.
\newblock Continuous-time average-preserving opinion dynamics with
  opinion-dependent communications.
\newblock {\em SIAM Journal on Control and Optimization}, 48(8):5214--5240,
  2010.

\bibitem{blume2015linear}
Lawrence~E Blume, William~A Brock, Steven~N Durlauf, and Rajshri Jayaraman.
\newblock Linear social interactions models.
\newblock {\em Journal of Political Economy}, 123(2):444--496, 2015.

\bibitem{borzi2015modeling}
Alfio Borzi and Suttida Wongkaew.
\newblock Modeling and control through leadership of a refined flocking system.
\newblock {\em Mathematical Models and Methods in Applied Sciences},
  25(02):255--282, 2015.

\bibitem{bryson2018applied}
Arthur~E Bryson and Yu-Chi Ho.
\newblock {\em Applied optimal control: optimization, estimation, and control}.
\newblock Routledge, 1975.

\bibitem{buechel2015opinion}
Berno Buechel, Tim Hellmann, and Stefan Kl{\"o}{\ss}ner.
\newblock Opinion dynamics and wisdom under conformity.
\newblock {\em Journal of Economic Dynamics and Control}, 52:240--257, 2015.

\bibitem{caillaud2007consensus}
Bernard Caillaud and Jean Tirole.
\newblock Consensus building: How to persuade a group.
\newblock {\em American Economic Review}, 97(5):1877--1900, 2007.

\bibitem{calvo2004effects}
Antoni Calvo-Armengol and Matthew~O Jackson.
\newblock The effects of social networks on employment and inequality.
\newblock {\em American Economic Review}, 94(3):426--454, 2004.

\bibitem{canuto2008eulerian}
Claudio Canuto, Fabio Fagnani, and Paolo Tilli.
\newblock A {E}ulerian approach to the analysis of rendez-vous algorithms.
\newblock {\em IFAC Proceedings Volumes}, 41(2):9039--9044, 2008.

\bibitem{che2009opinions}
Yeon-Koo Che and Navin Kartik.
\newblock Opinions as incentives.
\newblock {\em Journal of Political Economy}, 117(5):815--860, 2009.

\bibitem{ciarlet1989introduction}
Philippe~G Ciarlet, Bernadette Miara, and Jean-Marie Thomas.
\newblock {\em Introduction to numerical linear algebra and optimisation}.
\newblock Cambridge University Press, 1989.

\bibitem{degond2017continuum}
Pierre Degond, Jian-Guo Liu, Sara Merino-Aceituno, and Thomas Tardiveau.
\newblock Continuum dynamics of the intention field under weakly cohesive
  social interaction.
\newblock {\em Mathematical Models and Methods in Applied Sciences},
  27(01):159--182, 2017.

\bibitem{degroot1974reaching}
Morris~H DeGroot.
\newblock Reaching a consensus.
\newblock {\em Journal of the American Statistical Association},
  69(345):118--121, 1974.

\bibitem{demarzo2003persuasion}
Peter~M DeMarzo, Dimitri Vayanos, and Jeffrey Zwiebel.
\newblock Persuasion bias, social influence, and unidimensional opinions.
\newblock {\em The Quarterly Journal of Economics}, 118(3):909--968, 2003.

\bibitem{dietrich2017control}
Florian Dietrich, Samuel Martin, and Marc Jungers.
\newblock Control via leadership of opinion dynamics with state and
  time-dependent interactions.
\newblock {\em IEEE Transactions on Automatic Control}, 63(4):1200--1207, 2017.

\bibitem{douven2020mis}
Igor Douven and Rainer Hegselmann.
\newblock Mis- and disinformation in a bounded confidence model.
\newblock {\em Artificial Intelligence}, page 103415, 2020.

\bibitem{ellison1993rules}
Glenn Ellison and Drew Fudenberg.
\newblock Rules of thumb for social learning.
\newblock {\em Journal of Political Economy}, 101(4):612--643, 1993.

\bibitem{ellison1995word}
Glenn Ellison and Drew Fudenberg.
\newblock Word-of-mouth communication and social learning.
\newblock {\em The Quarterly Journal of Economics}, 110(1):93--125, 1995.

\bibitem{etesami2018influence}
S~Rasoul Etesami, Sadegh Bolouki, Angelia Nedi{\'c}, Tamer Ba{\c{s}}ar, and
  H~Vincent Poor.
\newblock Influence of conformist and manipulative behaviors on public opinion.
\newblock {\em IEEE Transactions on Control of Network Systems}, 6(1):202--214,
  2018.

\bibitem{festinger1957theory}
Leon Festinger.
\newblock {\em A theory of cognitive dissonance}, volume~2.
\newblock Stanford University Press, 1957.

\bibitem{french1956formal}
John~RP French~Jr.
\newblock A formal theory of social power.
\newblock {\em Psychological Review}, 63(3):181, 1956.

\bibitem{golub2010naive}
Benjamin Golub and Matthew~O Jackson.
\newblock Naive learning in social networks and the wisdom of crowds.
\newblock {\em American Economic Journal: Microeconomics}, 2(1):112--49, 2010.

\bibitem{gomes2005dynamic}
Armando Gomes and Philippe Jehiel.
\newblock Dynamic processes of social and economic interactions: On the
  persistence of inefficiencies.
\newblock {\em Journal of Political Economy}, 113(3):626--667, 2005.

\bibitem{han2019opinion}
Wenchen Han, Changwei Huang, and Junzhong Yang.
\newblock Opinion clusters in a modified hegselmann--krause model with
  heterogeneous bounded confidences and stubbornness.
\newblock {\em Physica A: Statistical Mechanics and its Applications},
  531:121791, 2019.

\bibitem{harary1959status}
Frank Harary.
\newblock Status and contrastatus.
\newblock {\em Sociometry}, 22(1):23--43, 1959.

\bibitem{hegselmann2017thomas}
Rainer Hegselmann.
\newblock Thomas {C}. {S}chelling and {J}ames {M}. {S}akoda: {T}he
  intellectual, technical, and social history of a model.
\newblock {\em Journal of Artificial Societies and Social Simulation}, 20(3),
  2017.

\bibitem{hegselmann2009deliberative}
Rainer Hegselmann and Ulrich Krause.
\newblock Deliberative exchange, truth, and cognitive division of labour: A
  low-resolution modeling approach.
\newblock {\em Episteme}, 6(2):130--144, 2009.

\bibitem{hegselmann2015opinion}
Rainer Hegselmann and Ulrich Krause.
\newblock Opinion dynamics under the influence of radical groups, charismatic
  leaders, and other constant signals: A simple unifying model.
\newblock {\em Networks \& Heterogeneous Media}, 10(3):477, 2015.

\bibitem{hegselmann2002opinion}
Rainer Hegselmann, Ulrich Krause, et~al.
\newblock Opinion dynamics and bounded confidence models, analysis, and
  simulation.
\newblock {\em Journal of Artificial Societies and Social Simulation}, 5(3),
  2002.

\bibitem{hegselmann2006truth}
Rainer Hegselmann, Ulrich Krause, et~al.
\newblock Truth and cognitive division of labor: First steps towards a computer
  aided social epistemology.
\newblock {\em Journal of Artificial Societies and Social Simulation}, 9(3):10,
  2006.

\bibitem{jabin2014clustering}
Pierre-Emmanuel Jabin and Sebastien Motsch.
\newblock Clustering and asymptotic behavior in opinion formation.
\newblock {\em Journal of Differential Equations}, 257(11):4165--4187, 2014.

\bibitem{kirchkamp2007naive}
Oliver Kirchkamp and Rosemarie Nagel.
\newblock Naive learning and cooperation in network experiments.
\newblock {\em Games and Economic Behavior}, 58(2):269--292, 2007.

\bibitem{leitmann1974cooperative}
George Leitmann.
\newblock {\em Cooperative and non-cooperative many players differential
  games}.
\newblock Springer, 1974.

\bibitem{mossel2020social}
Elchanan Mossel, Manuel Mueller-Frank, Allan Sly, and Omer Tamuz.
\newblock Social learning equilibria.
\newblock {\em Econometrica}, 88(3):1235--1267, 2020.

\bibitem{mueller2014does}
Manuel Mueller-Frank.
\newblock Does one {B}ayesian make a difference?
\newblock {\em Journal of Economic Theory}, 154:423--452, 2014.

\bibitem{nickerson1998confirmation}
Raymond~S Nickerson.
\newblock Confirmation bias: A ubiquitous phenomenon in many guises.
\newblock {\em Review of General Psychology}, 2(2):175--220, 1998.

\bibitem{pareto1964cours}
Vilfredo Pareto.
\newblock {\em Cours d'{\'e}conomie politique, Tome Premier}.
\newblock F. Rouge, Lausanne, 1896.

\bibitem{rosenberg2009informational}
Dinah Rosenberg, Eilon Solan, and Nicolas Vieille.
\newblock Informational externalities and emergence of consensus.
\newblock {\em Games and Economic Behavior}, 66(2):979--994, 2009.

\bibitem{rusinowska2019opinion}
Agnieszka Rusinowska and Akylai Taalaibekova.
\newblock Opinion formation and targeting when persuaders have extreme and
  centrist opinions.
\newblock {\em Journal of Mathematical Economics}, 84:9--27, 2019.

\bibitem{teschl2012ordinary}
Gerald Teschl.
\newblock {\em Ordinary differential equations and dynamical systems}, volume
  140.
\newblock American Mathematical Society, 2012.

\bibitem{wongkaew2015control}
Suttida Wongkaew, Marco Caponigro, and Alfio Borzi.
\newblock On the control through leadership of the {H}egselmann--{K}rause
  opinion formation model.
\newblock {\em Mathematical Models and Methods in Applied Sciences},
  25(03):565--585, 2015.

\end{thebibliography}

\end{document}